\documentclass[11pt]{article}
\usepackage{amsthm,amstext,amssymb,amsmath,graphicx,amsfonts,caption,subcaption,mathrsfs,multirow,dsfont}
\usepackage[utf8]{inputenc}
\usepackage{booktabs}
\usepackage{tikz}
    \usetikzlibrary{arrows,topaths}
    \usepackage{placeins}
\usepackage[mathscr]{euscript}
\usepackage{multicol}
\usepackage{natbib}
\usepackage[pagebackref=false,colorlinks,linkcolor=blue,citecolor=blue]{hyperref}

\theoremstyle{plain}
\newtheorem{theorem}{Theorem}[section]
\newtheorem{lemma}[theorem]{Lemma}
\newtheorem{proposition}[theorem]{Proposition}

\theoremstyle{definition}
\newtheorem{definition}[theorem]{Definition}
\newtheorem{corollary}[theorem]{Corollary}
\newtheorem{example}[theorem]{Example}

\theoremstyle{remark}
\newtheorem{remark}{\sc Remark}
\makeatother
\makeatletter
\def\namedlabel#1#2{\begingroup
   \def\@currentlabel{#2}%
   \label{#1}\endgroup}
\makeatother
\date{}
\title{\bf Gelfand residuated lattices}\vspace{.25 in}
\author{ \vspace{.25 in} {\bf Saeed Rasouli$^{1*}$} and {\bf Amin Dehghani $^2$}\\
Department of Mathematics, Persian Gulf University, \\Bushehr, Iran \\
{\tt $^1$srasouli@pgu.ac.ir }\\
{\tt $^2$dehghany.amin@hotmail.com}\\}
 \begin{document}
 \maketitle
 \begin{abstract}
In this paper, a combination of algebraic and topological methods is applied to obtain new and structural results on Gelfand residuated lattices. It is demonstrated that Gelfand's residuated lattices strongly tied up with the hull-kernel topology. Especially, it is shown that a residuated lattice is Gelfand if and only if its prime spectrum, equipped with the hull-kernel topology, is normal. The class of soft residuated lattices is introduced, and it is shown that a residuated lattice is soft if and only if it is Gelfand and semisimple. Gelfand residuated lattices are characterized using the pure part of filters. The relation between pure filters and radicals in a Gelfand residuated lattice is described. It is shown that a residuated lattice is Gelfand if and only if its pure spectrum is homeomorphic to its usual maximal spectrum. The pure filters of a Gelfand residuated lattice are characterized. Finally, it is proved that a residuated lattice is Gelfand if and only if the hull-kernel and the $\mathscr{D}$-topology coincide on the set of maximal filters.\footnote{2020 Mathematics Subject Classification: 06F99,06D20 \\
{\it Key words and phrases}: Gelfand residuated lattice; pure filters; soft residuated lattice, hull-kernel topology, $D$-topology.}
\end{abstract}
\section{Introduction}

\cite{gelfand1941normierte} showed that the functor from the category of compact Hausdorff spaces to the category of commutative $C^*$-algebras, obtained by assigning to each compact Hausdorff space $X$ the commutative $C^*$-algebra $C(X)$ of continuous complex functions on $X$, determines a duality between these categories. The dual functor is that obtained by assigning to each commutative $C^*$-algebra $\mathfrak{A}$ the compact Hausdorff space $Max(\mathfrak{A})$ of maximal ideals of $\mathfrak{A}$ endowed with the hull-kernel topology. To extend \textit{the Gelfand duality} to (not necessarily) commutative rings, \cite{mulvey1979generalisation} introduced the notion of a \textit{Gelfand ring}; a unitary ring $\mathfrak{A}$ that for any two distinct maximal right ideals $M$ and $M'$, there exist elements $a\notin M$ and
$a'\notin M'$ such that $aAa'=0$. He also asserted that the class of commutative Gelfand rings can be characterized by a property that can be formulated in
terms of universal algebra, namely that each prime ideal is contained in a unique maximal ideal. The class of rings that satisfy this universal property had been investigated by \cite{de1971commutative} under \textit{pm-rings}. Gelfand rings have been the main subject of many articles in the literature over the years and are still of current interest, see e.g. \cite{johnstone1982stone,contessa1982pm,contessa1984cerrain,al1989pure,al1990exchange,banaschewski2000gelfand,aghajani2020characterizations}. Motivated by \cite{de1971commutative} the notion of a \textit{pm-lattice} was introduced by \cite{pawar1977pm} as a bounded distributive lattice in which any prime ideal is contained in a unique maximal ideal. \citet[p. 185]{simmons1980reticulated} showed that a bounded distributive lattice is a pm-lattice if and only if it is a normal lattice. \citet[p. 199]{johnstone1982stone} proved that a commutative ring is a pm-ring if and only if its reticulation is a normal lattice. \citet[p. 208]{paseka1986regular} asserted that a unitary ring is Gelfand if and only if the set of its left ideals forms a normal quantale. \citet[p. 98]{georoescu1989some} characterized the class of normal multiplicative
ideal structures, as a generalization of some ordered algebras like the lattice of ideals in a Gelfand ring, the lattice of ideals
in a normal lattice, the lattice of ideals in an $F$-ring, the normal frames, and the normal quantales. MV-algebras \citep{filipoiu1995compact,cignoli2013algebraic}, BL-algebras \citep{kowalski2000residuated,di2003compact}, lattice-ordered groups \citep{bigard2006groupes}, normal residuated lattices \citep{bucsneag2012stable}, and normal quantales  \citep{cheptea2019boolean} are some classes of algebras that also satisfy the similar property. \citet[Definition 6.5]{georgescu2015algebraic} introduced the notion of \textit{Gelfand residuated lattices} and characterized them by the notion of reticulation. Recently, \citet[Proposition 8.8]{georgescu2020} proved that an algebra which satisfies some certain conditions is congruence normal if and only if any its prime congruence is contained in a unique maximal congruence.

Given the above discussions, we decided to take a deeper look at Gelfand residuated lattices. So notions of Gelfand residuated lattices are investigated, and some algebraic and topological characterizations are given for them. Our findings show that some results obtained by some above papers can also be reproduced via residuated lattices.

This paper is organized into four sections as follows: In Sect. \ref{sec2}, some definitions and facts about residuated lattices are recalled, and some of their propositions extracted. We illustrate this section with some examples of residuated lattices, which will be used in the following sections. Sect. \ref{sec3} deals with Gelfand residuated lattices. Theorem \ref{nococo} shows that a residuated lattice is Gelfand if and only if the bounded distributive lattice of its filters is normal. Theorem \ref{contchar} (Contessa's characterization) gives an element-wise characterization for Gelfand residuated lattices. The class of Gelfand residuated lattices is characterized by means of maximal filters in Theorem \ref{pmprop}. The remaining theorems of this section demonstrate that Gelfand residuated lattices are strongly tied up with the hull-kernel topology. Theorem \ref{gelsnor} shows that a residuated lattice is Gelfand if and only if its prime spectrum is normal. At the end of this section, following by \cite{bkouche1970couples}, the class of soft residuated lattices is introduced, and it is shown that a residuated lattice is soft if and only if it is Gelfand and semisimple. Sect. \ref{sec4} deals with the pure spectrum of a Gelfand residuated lattice. Theorem \ref{rfilter} characterizes Gelfand residuated lattices by means of the pure part of filters. In Theorem \ref{rhoradgel}, the relation between pure filters and radicals in a Gelfand residuated lattice is described. Theorem \ref{sppgelfch} verifies that a residuated lattice is Gelfand if and only if its pure spectrum is homeomorphic to its usual maximal spectrum. The pure filters of a Gelfand residuated lattice are characterized in Theorem \ref{mppurefcl}. Finally, Theorem \ref{rickhulldmin} proves that a residuated lattice is Gelfand if and only if its hull-kernel and its $\mathscr{D}$-topology coincide on the set of maximal filters.
\section{Residuated lattices}\label{sec2}

In this section, some definitions, properties, and results relative to residuated lattices, which will be used
in the following, recalled.

An algebra $\mathfrak{A}=(A;\vee,\wedge,\odot,\rightarrow,0,1)$ is called a \textit{residuated lattice} provided that $\ell(\mathfrak{A})=(A;\vee,\wedge,0,1)$ is a bounded lattice, $(A;\odot,1)$ is a commutative monoid, and $(\odot,\rightarrow)$ is an adjoint pair. A residuated lattice $\mathfrak{A}$ is called \textit{non-degenerate} if $0\neq 1$. On each residuated lattice $\mathfrak{A}$ we can define a unary operation $\neg$ by $\neg a:=a\rightarrow 0$. We also define
for all $a\in A$, $a^{1}=a$ and $a^{n+1}=a^{n}\odot a$, for any integer $n$. An element $a$ in $A$ is called \textit{idempotent} if $a^{2}=a$, and it is called \textit{nilpotent} if there exists an integer $n$ such that $a^{n}$=0. The minimum integer $n$ such that $a^{n}=0$ is called \textit{nilpotence order of $a$}. The set of nilpotent elements of $\mathfrak{A}$ is denoted by $\mathfrak{ni(A)}$. It is easy too see that $\mathfrak{ni(A)}$ is an ideal of $\ell(\mathfrak{A})$. In the following, we set $\mathfrak{in(A)}=A\setminus \mathfrak{ni(A)}$. The class of residuated lattices is equational, and so forms a variety. For a survey of residuated lattices, the reader is referred to \cite{galatos2007residuated}.
\begin{remark}\label{resproposition}\citep[Proposition 2.6]{ciungu2006classes}
Let $\mathfrak{A}$ be a residuated lattice. The following conditions are satisfied for any $x,y,z\in A$:
\begin{enumerate}
  \item [$r_{1}$ \namedlabel{res1}{$r_{1}$}] $x\odot (y\vee z)=(x\odot y)\vee (x\odot z)$;
  \item [$r_{2}$ \namedlabel{res2}{$r_{2}$}] $x\vee (y\odot z)\geq (x\vee y)\odot (x\vee z)$.
  \end{enumerate}
\end{remark}
\begin{example}\label{exa6}
Let $A_6=\{0,a,b,c,d,1\}$ be a lattice whose Hasse diagram is given by Figure \ref{figa6}. Routine calculation shows that $\mathfrak{A}_6=(A_6;\vee,\wedge,\odot,\rightarrow,0,1)$ is a residuated lattice in which the commutative operation $``\odot"$ is given by Table \ref{taba6} and the operation $``\rightarrow"$ is given by $x\rightarrow y=\bigvee \{a\in A_6|x\odot a\leq y\}$, for any $x,y\in A_6$.
\FloatBarrier
\begin{table}[ht]
\begin{minipage}[b]{0.56\linewidth}
\centering
\begin{tabular}{ccccccc}
\hline
$\odot$ & 0 & a & b & c & d & 1 \\ \hline
0       & 0 & 0 & 0 & 0 & 0 & 0 \\
        & a & a & a & 0 & a & a \\
        &   & b & a & 0 & a & b \\
        &   &   & c & c & c & c \\
        &   &   &   & d & d & d \\
        &   &   &   &   & 1 & 1 \\ \hline
\end{tabular}
\caption{Cayley table for ``$\odot$" of $\mathfrak{A}_6$}
\label{taba6}
\end{minipage}\hfill
\begin{minipage}[b]{0.6\linewidth}
\centering
  \begin{tikzpicture}[>=stealth',semithick,auto]
    \tikzstyle{subj} = [circle, minimum width=6pt, fill, inner sep=0pt]
    \tikzstyle{obj}  = [circle, minimum width=6pt, draw, inner sep=0pt]

    \tikzstyle{every label}=[font=\bfseries]

    \node[subj,  label=below:0] (0) at (0,0) {};
    \node[subj,  label=below:c] (c) at (-1,1) {};
    \node[subj,  label=below:a] (a) at (1,.5) {};
    \node[subj,  label=below right:b] (b) at (1,1.5) {};
    \node[subj,  label=below:d] (d) at (0,2) {};
    \node[subj,  label=below right:1] (1) at (0,3) {};

    \path[-]   (0)    edge                node{}      (a);
    \path[-]   (a)    edge                node{}      (b);
    \path[-]   (0)    edge                node{}      (c);
    \path[-]   (c)    edge                node{}      (d);
    \path[-]   (b)    edge                node{}      (d);
    \path[-]   (d)    edge                node{}      (1);
\end{tikzpicture}

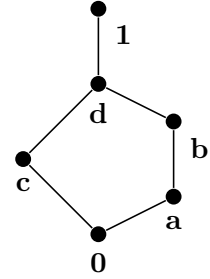
\captionof{figure}{Hasse diagram of $\mathfrak{A}_{6}$}
\label{figa6}
\end{minipage}
\end{table}
\FloatBarrier
\end{example}
\begin{example}\label{exa8}
Let $A_8=\{0,a,b,c,d,e,f,1\}$ be a lattice whose Hasse diagram is given by Figure \ref{figa8}. Routine calculation shows that $\mathfrak{A}_8=(A_8;\vee,\wedge,\odot,\rightarrow,0,1)$ is a residuated lattice in which the commutative operation $``\odot"$ is given by Table \ref{taba8} and the operation $``\rightarrow"$ is given by $x\rightarrow y=\bigvee \{a\in A_8|x\odot a\leq y\}$, for any $x,y\in A_8$.
\FloatBarrier
\begin{table}[ht]
\begin{minipage}[b]{0.56\linewidth}
\centering
\begin{tabular}{ccccccccc}
\hline
$\odot$ & 0 & a & b & c & d & e & f & 1 \\ \hline
0       & 0 & 0 & 0 & 0 & 0 & 0 & 0 & 0 \\
        & a & a & 0 & a & a & a & a & a \\
        &   & b & 0 & 0 & 0 & 0 & b & b \\
        &   &   & c & c & a & c & a & c \\
        &   &   &   & d & a & a & d & d \\
        &   &   &   &   & e & c & d & e \\
        &   &   &   &   &   & f & f & f \\
        &   &   &   &   &   &   & 1 & 1 \\ \hline
\end{tabular}
\caption{Cayley table for ``$\odot$" of $\mathfrak{A}_8$}
\label{taba8}
\end{minipage}\hfill
\begin{minipage}[b]{0.6\linewidth}
\centering
  \begin{tikzpicture}[>=stealth',semithick,auto]
    \tikzstyle{subj} = [circle, minimum width=6pt, fill, inner sep=0pt]
    \tikzstyle{obj}  = [circle, minimum width=6pt, draw, inner sep=0pt]

    \tikzstyle{every label}=[font=\bfseries]

    \node[subj,  label=below:0] (0) at (0,0) {};
    \node[subj,  label=below:a] (a) at (-1,1) {};
    \node[subj,  label=below:b] (b) at (1,1) {};
    \node[subj,  label=below:c] (c) at (-2,2) {};
    \node[subj,  label=below:d] (d) at (0,2) {};
    \node[subj,  label=below:e] (e) at (-1,3) {};
    \node[subj,  label=below:f] (f) at (1,3) {};
    \node[subj,  label=below:1] (1) at (0,4) {};

    \path[-]   (0)    edge                node{}      (a);
    \path[-]   (0)    edge                node{}      (b);
    \path[-]   (b)    edge                node{}      (d);
    \path[-]   (d)    edge                node{}      (f);
    \path[-]   (f)    edge                node{}      (1);
    \path[-]   (a)    edge                node{}      (d);
    \path[-]   (a)    edge                node{}      (c);
    \path[-]   (c)    edge                node{}      (e);
    \path[-]   (d)    edge                node{}      (e);
    \path[-]   (e)    edge                node{}      (1);
\end{tikzpicture}

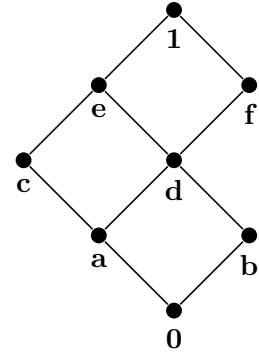
\captionof{figure}{Hasse diagram of $\mathfrak{A}_{8}$}
\label{figa8}
\end{minipage}
\end{table}
\FloatBarrier
\end{example}

Let $\mathfrak{A}$ be a residuated lattice. A non-void subset $F$ of $A$ is called a \textit{filter} of $\mathfrak{A}$ provided that $x,y\in F$ implies $x\odot y\in F$, and $x\vee y\in F$, for any $x\in F$ and $y\in A$. The set of filters of $\mathfrak{A}$ is denoted by $\mathscr{F}(\mathfrak{A})$. A filter $F$ of $\mathfrak{A}$ is called \textit{proper} if $F\neq A$. For any subset $X$ of $A$, the \textit{filter of $\mathfrak{A}$ generated by $X$} is denoted by $\mathscr{F}(X)$. For each $x\in A$, the filter generated by $\{x\}$ is denoted by $\mathscr{F}(x)$ and said to be \textit{principal}. The set of principal filters is denoted by $\mathscr{PF}(\mathfrak{A})$. Following \citet[\S 5.7]{gratzer2011lattice}, a join-complete lattice $\mathfrak{A}$, is called a \textit{frame} if it satisfies the join infinite distributive law (JID), i.e., for any $a\in A$ and $S\subseteq A$, $a\wedge \bigvee S=\bigvee \{a\wedge s\mid s\in S\}$. A frame $\mathfrak{A}$ is called complete provided that $\mathfrak{A}$ is a complete lattice. According to \cite{galatos2007residuated}, $(\mathscr{F}(\mathfrak{A});\cap,\veebar,\textbf{1},A)$ is a complete frame, in which $\veebar \mathcal{F}=\mathscr{F}(\cup \mathcal{F})$, for any $\mathcal{F}\subseteq \mathscr{F}(\mathfrak{A})$.
\begin{example}\label{filterexa}
Consider the residuated lattice $\mathfrak{A}_6$ from Example \ref{exa6} and the residuated lattice $\mathfrak{A}_8$ from Example \ref{exa8}. The sets of their filters are presented in Table \ref{tafiex}.
\begin{table}[h]
\centering
\begin{tabular}{ccl}
\hline
                 & \multicolumn{2}{c}{Filters}                                       \\ \hline
$\mathfrak{A}_6$ & \multicolumn{2}{c}{$\{1\},\{a,b,d,1\},\{c,d,1\},\{d,1\},A_6$} \\
$\mathfrak{A}_8$ & \multicolumn{2}{c}{$\{1\},\{a,c,d,e,f,1\},\{c,e,1\},\{f,1\},A_8$} \\ \hline
\end{tabular}
\caption{The sets of filters of $\mathfrak{A}_6$ and $\mathfrak{A}_8$}
\label{tafiex}
\end{table}
\end{example}

The proof of the following proposition has a routine verification, and so it is left to the reader.
\begin{proposition}\label{genfilprop}
Let $\mathfrak{A}$ be a residuated lattice and $F$ be a filter of $\mathfrak{A}$. The following assertions hold, for any $x,y\in A$:
\begin{enumerate}
  \item  [(1) \namedlabel{genfilprop1}{(1)}] $\mathscr{F}(F,x)\stackrel{def.}{=}F\veebar \mathscr{F}(x)=\{a\in A|f\odot x^n\leq a,\textrm{~for~some}~f\in F~\textrm{and}~n\in \mathbb{N}\}$;
  \item  [(2) \namedlabel{genfilprop2}{(2)}] $x\leq y$ implies $\mathscr{F}(y)\subseteq \mathscr{F}(x)$.
  \item  [(3) \namedlabel{genfilprop3}{(3)}] $\mathscr{F}(x)\cap \mathscr{F}(y)=\mathscr{F}(x\vee y)$;
  \item  [(4) \namedlabel{genfilprop4}{(4)}] $\mathscr{F}(x)\veebar \mathscr{F}(y)=\mathscr{F}(x\odot y)$;
  \item  [(5) \namedlabel{genfilprop5}{(5)}] $\mathscr{PF}(\mathfrak{A})$ is a sublattice of $\mathscr{F}(\mathfrak{A})$.
\end{enumerate}
\end{proposition}

The following proposition gives a characterization for the comaximal filters of a residuated lattice.
\begin{proposition}\label{compropo}
  Let $\mathfrak{A}$ be a residuated lattice and $F,G$ two proper filters of $\mathfrak{A}$. The following assertions are equivalent:
  \begin{enumerate}
  \item  [$(1)$ \namedlabel{compropo1}{$(1)$}] $F$ and $G$ are comaximal, i.e., $F\veebar G=A$;
  \item  [$(2)$ \namedlabel{compropo2}{$(2)$}] there exist $f\in F$ and $g\in G$ such that $f\odot g=0$;
  \item  [$(3)$ \namedlabel{compropo3}{$(3)$}] there exists $a\in A$ such that $a\in F$ and $\neg a\in G$.
\end{enumerate}
\end{proposition}
\begin{proof}
\item [\ref{compropo1}$\Rightarrow$\ref{compropo2}:] It is evident by Proposition \ref{genfilprop}.
\item [\ref{compropo2}$\Rightarrow$\ref{compropo3}:] Let $f\odot g=0$, for some $f\in F$ and $g\in G$. This implies that $g\leq \neg f$, and this holds the result.
\item [\ref{compropo3}$\Rightarrow$\ref{compropo1}:] It is evident.
\end{proof}

A proper filter of a residuated lattice $\mathfrak{A}$ is called \textit{maximal} if it is a maximal element in the set of all proper filters. The set of maximal filters of $\mathfrak{A}$ is denoted by $Max(\mathfrak{A})$. Zorn's lemma verifies that any proper filter is contained in a maximal filter.
\begin{proposition}\cite[Theorem 3.31]{ciungu2006classes}\label{maxinufilpro}
Let $\mathfrak{A}$ be a residuated lattice and $\mathfrak{m}$ be a proper filter of $\mathfrak{A}$. Then $\mathfrak{m}$ is maximal if and only if for any $x\notin \mathfrak{m}$ there exists an integer $n$ such that $\neg x^n\in \mathfrak{m}$.
\end{proposition}
Let $F$ be a filter of a residuated lattice $\mathfrak{A}$. The intersection of all maximal filters of $\mathfrak{A}$ containing $F$ is denoted by $Rad(F)$. It is well-known that $\mathfrak{A}$ is  \textit{semisimple} algebra, i.e.  it is a subdirect product of simple algebras, provided that $Rad(\mathfrak{A})=\{1\}$.
\begin{proposition}\cite[Theorem 4.3]{freytes2008injectives}\label{radchar}
For any residuated lattice $\mathfrak{A}$ we have $Rad(\mathfrak{A})=\{a\in A\mid a~is~unity\}$.
\end{proposition}

Following \citet[Definition 4.1]{ciungu2006classes}, a residuated lattice $\mathfrak{A}$ is called \textit{local} provided that it has exactly one maximal filter.
\begin{proposition}\label{ciucharlocal}\cite[Proposition 4.4]{ciungu2006classes}
  Let $\mathfrak{A}$ be a residuated lattice. The following assertions are equivalent:
\begin{enumerate}
   \item [(1) \namedlabel{ciucharlocal1}{(1)}] $\mathfrak{A}$ is local;
   \item [(2) \namedlabel{ciucharlocal2}{(2)}] $\mathfrak{in}(\mathfrak{A})$ is a filter of $\mathfrak{A}$;
   \item [(3) \namedlabel{ciucharlocal3}{(3)}] $\mathfrak{in}(\mathfrak{A})$ is a proper filter of $\mathfrak{A}$;
   \item [(4) \namedlabel{ciucharlocal4}{(4)}] $\mathfrak{in}(\mathfrak{A})$ is the unique maximal filter of $\mathfrak{A}$;
   \item [(5) \namedlabel{ciucharlocal5}{(5)}] $x\odot y\in \mathfrak{ni}(\mathfrak{A})$ implies that $x\in \mathfrak{ni}(\mathfrak{A})$ or $y\in \mathfrak{ni}(\mathfrak{A})$, for any $x,y\in A$.
 \end{enumerate}
\end{proposition}

A proper filter $\mathfrak{p}$ of $\mathfrak{A}$ is called \textit{prime} if $x\vee y\in \mathfrak{p}$ implies $x\in \mathfrak{p}$ or $y\in \mathfrak{p}$, for any $x,y\in A$. The set of prime filters of $\mathfrak{A}$ is denoted by $Spec(\mathfrak{A})$. Since $\mathscr{F}(\mathfrak{A})$ is a distributive lattice, so $Max(\mathfrak{A})\subseteq Spec(\mathfrak{A})$.

A non-empty subset $\mathscr{C}$ of $\mathfrak{A}$ is called \textit{$\vee$-closed} if it is closed under the join operation, i.e $x,y\in \mathscr{C}$ implies $x\vee y\in \mathscr{C}$.
\begin{theorem}\cite[Theorem 3.18]{rasouli2019going}\label{prfilth}
If $\mathscr{C}$ is a $\vee$-closed subset of $\mathfrak{A}$ which does not meet the filter $F$, then $F$ is contained in a filter $\mathfrak{p}$ which is maximal with respect to the property of not meeting $\mathscr{C}$; furthermore $\mathfrak{p}$ is prime.
\end{theorem}

\begin{example}\label{maxminex}
Consider the residuated lattice $\mathfrak{A}_6$ from Example \ref{exa6} and the residuated lattice $\mathfrak{A}_8$ from Example \ref{exa8}. The sets of their maximal and prime filters are presented in Table \ref{prfiltab}.
\begin{table}[h]
\centering
\begin{tabular}{ccc}
\hline
                 & \multicolumn{2}{c}{Prime filters}      \\ \hline
                 & Maximal filters &                \\
$\mathfrak{A}_6$ &  $\{a,b,d,1\},\{c,d,1\}$                 & $\{d,1\},\{1\}$\\
$\mathfrak{A}_8$ &  $\{a,c,d,e,f,1\}$                          &         $\{c,e,1\},\{f,1\}$ \\ \hline
\end{tabular}
\caption{The sets of maximal and prime filters of $\mathfrak{A}_6$ and $\mathfrak{A}_8$}
\label{prfiltab}
\end{table}
\end{example}

Let $\mathfrak{A}$ be a residuated lattice and $\Pi$ a collection of prime filters of $\mathfrak{A}$.
For a subset $\pi$ of $\Pi$ we set $k(\pi)=\bigcap \pi$, and for a subset $X$ of $A$ we set $h_{\Pi}(X)=\{P\in \Pi\mid X\subseteq  P\}$ and $d_{\Pi}(X)=\Pi\setminus h_{\Pi}(X)$. The collection $\Pi$ can be topologized by taking the collection $\{h_{\Pi}(x)\mid x\in A\}$  as a closed (an open) basis, which is called \textit{the (dual) hull-kernel topology} on $\Pi$ and denoted by $\Pi_{h(d)}$. Also, the generated
topology by $\tau_{h}\cup \tau_{d}$ is called \textit{the patch topology} and denoted by $\tau_{p}$. For a subset $X$ of $A$, we set $H_{\Pi}(X)=\{h_{\Pi}(x)\mid x\in X\}$. As usual, the Boolean lattice of all clopen subsets of a topological space $A_{\tau}$ shall be denoted by $Clop(A_{\tau})$. For a detailed discussion on the (dual) hull-kernel and patch topologies on a residuated lattice, we refer to \cite{rasouli2018hull}.
\begin{proposition}\citep{rasouli2018hull}\label{minpro}
Let $\mathfrak{A}$ be a residuated lattice. We have:
\[Clop(Spec_{h}(\mathfrak{A}))=H(\beta(\mathfrak{A}))\]
\end{proposition}

\begin{proposition}\citep{rasouli2018hull}\label{hulkerinstr}
Let $\mathfrak{A}$ be a residuated lattice. The following assertions hold:
\begin{enumerate}
\item  [$(1)$ \namedlabel{hulkerinstr1}{$(1)$}] $\mathfrak{A}$ is semisimple if and only if $Max(\mathfrak{A})$ is dense in $Spec_{h}(\mathfrak{A})$;
\item  [$(2)$ \namedlabel{hulkerinstr2}{$(2)$}] $Spec_{h}(\mathfrak{A})$ and $Max_{h}(\mathfrak{A})$ are compact.
\end{enumerate}
\end{proposition}

Let $\Pi$ be a collection of prime filters in a residuated lattice $\mathfrak{A}$. Following \citet[p. 290]{de1983projectivity}, if $\pi$ is a subset of $\Pi$, its \textit{specialization} (\textit{generalization}) in $\Pi$, $\mathscr{S}_{\Pi}(\pi)$ ($\mathscr{G}_{\Pi}(\pi)$), is the set of all primes in $\Pi$, which contain (are contained in) some prime belonging to $\pi$. One can see that $\mathscr{S}$ and $\mathscr{G}$ are closure operators on the power set of $Spec(\mathfrak{A})$. A fixed point of $\mathscr{S(G)}$ is called $\mathscr{S}_{\Pi}$-stable ($\mathscr{G}_{\Pi}$-stable). If $\Pi$ is understood, it will be dropped. Notice that for any subset $B$ of $A$, $\bigcup_{b\in B}h(b)(\bigcup_{b\in B}d(b))$ is $\mathscr{S(G)}$-stable. In the following, for a given subset $\pi$ of $\Pi$, $cl^{\Pi}_{h(d)}(\pi)$ stands for the closure of $\pi$ in the topological space $\Pi_{h(d)}$. If $\Pi$ is understood, it will be dropped.
\begin{lemma}\citep[Theorem 3.14]{rasouli2018hull}\label{retractlemma}
Let $\mathfrak{A}$ be a residuated lattice, $\Pi$ a collection of prime filters of $\mathfrak{A}$, and $\pi\subseteq \Pi$. The following assertions hold:
\begin{enumerate}
  \item [$(1)$ \namedlabel{retractlemma1}{$(1)$}] $cl_{h}(\pi)=hk(\pi)$;
  \item [$(2)$ \namedlabel{retractlemma2}{$(2)$}] if $\pi$ is compact w.r.t the dual hull-kernel topology, then $cl_{h}(\pi )=\mathscr{S}(\pi)$;
  \item [$(3)$ \namedlabel{retractlemma3}{$(3)$}] if $\pi$ is compact w.r.t the hull-kernel topology, then $cl_{d}(\pi )=\mathscr{G}(\pi)$.
\end{enumerate}
\end{lemma}

The following theorem characterizes the closed sets of the hull-kernel topology.
\begin{theorem}\cite[Theorem 4.30]{rasouli2018hull}\label{closefalzai}
  Let $\mathfrak{A}$ be a residuated lattice and $\pi$ a subset of $Spec(\mathfrak{A})$. $\pi$ is closed under the hull-kernel topology if and only if it is closed under the patch topology and $\mathscr{S}$-stable.
\end{theorem}

Let $\mathfrak{A}$ be a residuated lattice. For any subset $X$ of $A$, we set $X^{\perp}=kd(X)$, $\Gamma(\mathfrak{A})=\{X^{\perp}|X\subseteq A\}$, and $\gamma(\mathfrak{A})=\{x^{\perp}|x\in A\}$. Elements of $\Gamma(\mathfrak{A})$ and $\gamma(\mathfrak{A})$ are called \textit{coannihilators} and \textit{coannulets} of $\mathfrak{A}$, respectively. By \citet[Proposition 3.13]{rasouli2018generalized} follows that $(\Gamma(\mathfrak{A});\cap,\vee^{\Gamma},\{1\},A)$ is a complete Boolean lattice in which for any $\mathscr{F}\subseteq \Gamma(\mathfrak{A})$ we have $\vee^{\Gamma} \mathscr{F}=(\cup \mathscr{F})^{\perp\perp}$. By \citet[Proposition 2.15]{rasouli2020quasicomplemented} follows that $\gamma(\mathfrak{A})$ is a sublattice of $\Gamma(\mathfrak{A})$.  $\mathfrak{A}$ shall be called \textit{Baer} provided that $\Gamma(\mathfrak{A})$ is a sublattice of $\mathscr{F}(\mathfrak{A})$, and \textit{Rickart} provided that $\gamma(\mathfrak{A})$ is a Boolean sublattice of $\mathscr{F}(\mathfrak{A})$. For the basic facts concerning coannihilators and coannulets of residuated lattices we refer to \cite{rasouli2018generalized}.

    Let $\mathfrak{A}$ be a residuated lattice. For an ideal $I$ of $\ell(\mathfrak{A})$, set $\omega(I)=\{a\in A|a\vee x=1,\textrm{~for~some~} x\in I\}$, and $\Omega(\mathfrak{A})=\{\omega(I)|I\in \ell(\mathfrak{A})\}$. Using Proposition 3.4 of \cite{rasouli2018n}, it follows that $\Omega(\mathfrak{A})\subseteq \mathscr{F}(\mathfrak{A})$, and so elements of $\Omega(\mathfrak{A})$ are called \textit{$\omega$-filters} of $\mathfrak{A}$. For an $\omega$-filter $F$ of $\mathfrak{A}$, $I_{F}$ denoted an ideal of $\ell(\mathfrak{A})$, which satisfies $F=\omega(I_{F})$. \citet[Proposition 3.7]{rasouli2018n} show that $(\Omega(\mathfrak{A});\cap,\vee^{\omega},\{1\},A)$ is a bounded distributive lattice, in which $F\vee^{\omega} G=\omega(I_{F}\curlyvee I_{G})$, for any $F,G\in \Omega(\mathfrak{A})$ (by $\curlyvee$, we mean the join operation in the lattice of ideals of $\ell(\mathfrak{A})$). For a prime filter $\mathfrak{p}$ of $\mathfrak{A}$, $\omega(\dot{\mathfrak{p}})$ is called the $D$-part of $\mathfrak{p}$ and denoted by $D(\mathfrak{p})$. For the basic facts concerning $\omega$-filters of a residuated lattice we refer to \cite{rasouli2018n}.
\begin{proposition}\citep{rasouli2018n}\label{omegprop}
    Let $\mathfrak{A}$ be residuated lattice. The following assertions hold:
    \begin{enumerate}
\item  [$(1)$ \namedlabel{omegprop1}{$(1)$}] $\gamma(\mathfrak{A})$ is a sublattice of $\Omega(\mathfrak{A})$;
\item  [$(2)$ \namedlabel{omegprop2}{$(2)$}] $D(\mathfrak{p})=k(\mathscr{G}(\mathfrak{p}))$, for any $\mathfrak{p}\in Spec(\mathfrak{A})$;
\item  [$(3)$ \namedlabel{omegprop3}{$(3)$}] $\bigcap_{\mathfrak{m}\in Max(\mathfrak{A})}D(\mathfrak{m})=\{1\}$.
\end{enumerate}
  \end{proposition}
  \begin{proposition}\label{dabeta}
    Let $\mathfrak{A}$ be residuated lattice. The following assertions hold:
    \begin{enumerate}
\item  [$(1)$ \namedlabel{dabeta1}{$(1)$}] $\beta(\mathfrak{A})\cap \mathfrak{p}\subseteq D(\mathfrak{p})$, for any $\mathfrak{p}\in Spec(\mathfrak{A})$;
\item  [$(2)$ \namedlabel{dabeta2}{$(2)$}] $\beta(\mathfrak{A})\cap Rad(\mathfrak{A})=\{1\}$;
\item  [$(3)$ \namedlabel{dabeta3}{$(3)$}] $D(\mathfrak{p})\veebar D(\mathfrak{q})=A$ if and only if $\dot{\mathfrak{p}}\curlyvee\dot{\mathfrak{q}}=A$, for any $\mathfrak{p},\mathfrak{q}\in Spec(\mathfrak{A})$.
\end{enumerate}
  \end{proposition}
\begin{proof}
  \item [\ref{dabeta1}:] It is evident.
  \item [\ref{dabeta2}:] It follows by Proposition \ref{omegprop}\ref{omegprop3} and \ref{dabeta1}.
  \item [\ref{dabeta3}:] It follows by \citep[Proposition 3.5]{rasouli2018n}.
\end{proof}

\begin{proposition}\label{cormaxqout}\cite[Corollary 3.9]{rasouli2019going}
Let $\mathfrak{A}$ be a residuated lattice and $F$ a filter of $\mathfrak{A}$. Then
\[Max(\mathfrak{A}/F)=\{\mathfrak{m}/F|\mathfrak{m}\in h(F)\cap Max(\mathfrak{A})\}.\]
\end{proposition}
\begin{proposition}\label{localqoiemax}
  Let $\mathfrak{A}$ be a residuated lattice and $\mathfrak{m}$ a maximal filter of $\mathfrak{A}$. The following assertions are equivalent:
\begin{enumerate}
   \item [(1) \namedlabel{localqoiemax1}{(1)}] $\mathfrak{A}/D(\mathfrak{m})$ is local;
   \item [(2) \namedlabel{localqoiemax2}{(2)}] for any $x\notin \mathfrak{m}$, there exist an integer $n$ and $a\notin \mathfrak{m}$ such that $a\vee \neg x^{n}=1$.
 \end{enumerate}
\end{proposition}
\begin{proof}
By Proposition \ref{cormaxqout} follows that $\mathfrak{m}/D(\mathfrak{m})$ is a maximal filter of $\mathfrak{A}/D(\mathfrak{m})$.
  \item [\ref{localqoiemax1}$\Rightarrow$\ref{localqoiemax2}:] Applying Proposition \ref{ciucharlocal}, it follows that $\mathfrak{m}/D(\mathfrak{m})=\mathfrak{in}(\mathfrak{A}/D(\mathfrak{m}))$. Consider $x\notin \mathfrak{m}$. So $\neg x^{n}\in D(\mathfrak{m})$, for some integer $n$. This holds the result.
  \item [\ref{localqoiemax2}$\Rightarrow$\ref{localqoiemax1}:] Routinely, one can show that $\mathfrak{m}/D(\mathfrak{m})=\mathfrak{in}(\mathfrak{A}/D(\mathfrak{m}))$. Thus the result holds by Proposition \ref{ciucharlocal}.
\end{proof}
\section{Gelfand residuated lattices}\label{sec3}

In this section, the notion of a Gelfand residuated lattice is investigated, and some algebraic and topological characterizations of them are extracted.


\begin{definition}
A residuated lattice $\mathfrak{A}$ is called \textit{Gelfand} provided that any prime filter of $\mathfrak{A}$ is contained in a unique maximal filter of $\mathfrak{A}$.
\end{definition}
\begin{example}\label{quanorexas}
One can see that the residuated lattice $\mathfrak{A}_6$ from Example \ref{exa6} is not Gelfand and the residuated lattice $\mathfrak{A}_8$ from Example \ref{exa8} is Gelfand.
\end{example}
\begin{example}
   The class of MTL-algebras, and so,  MV-algebras, BL-algebras, and Boolean algebras are some subclasses of Gelfand residuated lattices.
\end{example}
Let $\mathfrak{A}$ be a bounded distributive lattice. $\mathfrak{A}$ is said to be:
\begin{itemize}
  \item \textit{normal} provided that for all $x,y\in A$, $x\vee y=1$ implies there exist $u,v\in A$ such that $u\vee x=v\vee y=1$ and $u\wedge v=0$;
  \item \textit{conormal} provided that for all $x,y\in A$, $x\wedge y=0$ implies there exist $u,v\in L$ such that $u\wedge x=v\wedge y=0$ and $u\vee v=1$.
\end{itemize}
\begin{remark}
In \cite{cornish1972normal} and \cite{pawar1994characterizations}, the above nomenclatures are reversed. We have picked the version of these definitions from \citet[Definition 4.3]{simmons1980reticulated} and \citet[p. 67]{johnstone1982stone} because of the author's discussion in \cite[p. 78]{johnstone1982stone}.
\end{remark}

The following result shows that a residuated lattice is Gelfand if and only if the bounded distributive lattice of its filters is normal.
\begin{theorem}\label{nococo}
Let $\mathfrak{A}$ be a residuated lattice. The following assertions are equivalent:
\begin{enumerate}
\item  [$(1)$ \namedlabel{nococo1}{$(1)$}] The bounded distributive lattice $\mathscr{F}(\mathfrak{A})$ is normal;
\item  [$(2)$ \namedlabel{nococo2}{$(2)$}] the bounded distributive lattice $\mathscr{PF}(\mathfrak{A})$ is normal;
\item  [$(3)$ \namedlabel{nococo3}{$(3)$}] $\mathfrak{A}$ is Gelfand.
\end{enumerate}
\end{theorem}
\begin{proof}
\item [\ref{nococo1}$\Rightarrow$\ref{nococo2}:] Let $x,y\in A$, such that $\mathscr{F}(x)\veebar \mathscr{F}(y)=A$. Then there exist $F,G\in \mathscr{F}(\mathfrak{A})$ such that $F\cap G=\{1\}$ and $F\veebar \mathscr{F}(x)=G\veebar \mathscr{F}(y)=A$. Thus there exist $f\in F$, $g\in G$, and integer $n$ such that $f\odot x^{n}=g\odot y^{n}=0$. This implies that $\mathscr{F}(f)\cap \mathscr{F}(g)=\{1\}$ and $\mathscr{F}(f)\veebar \mathscr{F}(x)=\mathscr{F}(g)\veebar \mathscr{F}(y)=A$.
\item [\ref{nococo2}$\Rightarrow$\ref{nococo3}:] With a little bit of effort, it proves by Proposition \ref{genfilprop} and Theorem \ref{contchar}.
\item [\ref{nococo3}$\Rightarrow$\ref{nococo1}:] Let $F$ and $G$ be two filters of $\mathfrak{A}$ such that $F\veebar G=A$. So there exist some $f\in F$ and $g\in G$ such that $f\odot g=0$. This establishes the result by taking $H=\mathscr{F}(\neg f^{n})$ and $K=\mathscr{F}(\neg g^{n})$ with integer $n$ obtained from Theorem \ref{contchar}.
\end{proof}
\citet[Theorem 4.1]{contessa1982pm} gave an element-wise criterion for a ring to be Gelfand and showed that a unitary commutative ring $\mathfrak{A}$ is Gelfand if and only if it satisfies a certain condition:
\begin{center}
$\forall m\in A$, $\exists a,b\in A$ such that $(1-am)(1-bm')=1$, $m'=1-m$.
\end{center}

Motivated by this, the following theorem generalizes this criterion, and gives an element-wise characterization for Gelfand residuated lattices.
\begin{theorem}\label{contchar}(Contessa's characterization)
Let $\mathfrak{A}$ be a residuated lattice. The following assertions are equivalent:
\begin{enumerate}
   \item [(1) \namedlabel{contchar1}{(1)}] $\mathfrak{A}$ is Gelfand;
   \item [(2) \namedlabel{contchar2}{(2)}] for any $a,b\in A$ with $a\odot b=0$, there exist some integers $m,n$ such that $\neg a^m\vee \neg b^n=1$.
 \end{enumerate}
\end{theorem}
\begin{proof}
\item [\ref{contchar1}$\Rightarrow$\ref{contchar2}:] Let $a,b\in A$ such that $a\odot b=0$. It is obvious that $S=\{\neg a^{m}\vee \neg b^{n}\mid m,n\in \mathds{N}\cup \{0\}\}$ is a $\vee$-closed subset of $\mathfrak{A}$. Assume by absurdum that  $1\notin S$. By Theorem \ref{prfilth}, there exists $\mathfrak{p}\in Spec(\mathfrak{A})$ which is maximal with respect to the property of not meeting $S$. The filter $\mathscr{F}(\mathfrak{p},a)$ is not all $A$ because the equality $p\odot a^{m}=0$, $p\in \mathfrak{p}$, $m\in \mathds{N}$, would implies $p\leq \neg a^{m}\in \mathfrak{p}\cap S$ which is impossible. So there exists a maximal filter $\mathfrak{m}_{a}$ which contains $\mathscr{F}(\mathfrak{p},a)$. Analogously, there exists a maximal filter  $\mathfrak{m}_{b}$ which contains $\mathscr{F}(\mathfrak{p},b)$. Notice that $\mathfrak{m}_{a}\neq \mathfrak{m}_{b}$, otherwise $0=a\odot b\in \mathfrak{m}_{a}=\mathfrak{m}_{b}$ which is impossible. In conclusion, $\mathfrak{p}$ is contained in two distinct maximal filters which is a contradiction. So $1\in S$ and this holds the result.
\item [\ref{contchar2}$\Rightarrow$\ref{contchar1}:] Let $\mathfrak{p}$ be a prime filter containing in two distinct maximal filters $\mathfrak{m}$ and $\mathfrak{n}$. So there exists $a\in \mathfrak{m}$ and $b\in \mathfrak{n}$ such that $a\odot b=0$. Thus we have $\neg a^m\vee \neg b^n=1\in \mathfrak{p}$, for some integers $m,n$. It follows then either $\neg a^m\in P$ or $\neg b^n\in P$, and then either $\mathfrak{m}=A$ or $\mathfrak{n}=A$; a contradiction.
\end{proof}
\begin{corollary}\label{dirgel}
  A finite direct product of Gelfand residuated lattices is Gelfand.
\end{corollary}
\begin{proof}
  Let $\{\mathfrak{A}_{i}\}_{i=1}^{s}$ be a finite family of Gelfand residuated lattices and set $\mathfrak{A}=\prod_{i=1}^{s}\mathfrak{A}_{i}$. Let $a,b\in A$ such that $a\odot b=0$. Then $a_{i}\odot b_{i}=0$ in $\mathfrak{A}_{i}$ implies the existence of integers $m_{i}$ such that $\neg a_{i}^{m_{i}}\vee \neg b_{i}^{m_{i}}=1$, for all $i\in \{1,\cdots,s\}$. By considering $m=\max\{m_1,\cdots,m_s\}$, we have $\neg a^{m}\vee \neg b^{n}=1$ and this holds the result by Proposition \ref{contchar}.
\end{proof}

The following theorem gives some algebraic characterizations by means of maximal filters for a Gelfand residuated lattice.
\begin{theorem}\label{pmprop}
Let $\mathfrak{A}$ be a residuated lattice.  The following assertions are equivalent:
 \begin{enumerate}
   \item [(1) \namedlabel{pmprop1}{(1)}] $\mathfrak{A}$ is Gelfand;
   \item [(2) \namedlabel{pmprop2}{(2)}] for any distinct maximal filters $\mathfrak{m}$ and $\mathfrak{n}$ of $\mathfrak{A}$, there exist $a\notin \mathfrak{m}$ and $b\notin \mathfrak{n}$ such that $a\vee b=1$;
   \item [(3) \namedlabel{pmprop3}{(3)}] for any distinct maximal filters $\mathfrak{m}$ and $\mathfrak{n}$ of $\mathfrak{A}$, $D(\mathfrak{m})$ and $D(\mathfrak{n})$ are comaximal;
   \item [(4) \namedlabel{pmprop4}{(4)}] for any distinct maximal filters $\mathfrak{m}$ and $\mathfrak{n}$ of $\mathfrak{A}$, there exists $a\in A$ such that $a\in D(\mathfrak{m})$ and $\neg a\in D(\mathfrak{n})$;
   \item [(5) \namedlabel{pmprop5}{(5)}] for any maximal filter $\mathfrak{m}$ and any prime filter $\mathfrak{p}$ of $\mathfrak{A}$, $D(\mathfrak{p})\subseteq \mathfrak{m}$ implies $\mathfrak{p}\subseteq \mathfrak{m}$;
   \item [(6) \namedlabel{pmprop6}{(6)}] for any maximal filter $\mathfrak{m}$ and any proper filter $F$ of $\mathfrak{A}$, $D(\mathfrak{m})\subseteq F$ implies $F\subseteq \mathfrak{m}$;
   \item [(7) \namedlabel{pmprop7}{(7)}] for any maximal filter $\mathfrak{m}$ and any proper filter $F$ of $\mathfrak{A}$, $F\veebar \mathfrak{m}=A$ implies $F\veebar D(\mathfrak{m})=A$;
   \item [(8) \namedlabel{pmprop8}{(8)}] for any maximal filter $\mathfrak{m}$ of $\mathfrak{A}$, $\mathfrak{m}$ is the unique maximal filter containing $D(\mathfrak{m})$;
   \item [(9) \namedlabel{pmprop9}{(9)}] for any maximal filter $\mathfrak{m}$, $\mathscr{G}(\mathfrak{m})=h(D(\mathfrak{m}))$;
   \item [(10) \namedlabel{pmprop10}{(10)}] for any maximal filter $\mathfrak{m}$ of $\mathfrak{A}$, $\mathfrak{A}/D(\mathfrak{m})$ is local;
   \item [(11) \namedlabel{pmprop11}{(11)}] for any maximal filter $\mathfrak{m}$ of $\mathfrak{A}$ and any $x\notin \mathfrak{m}$, there exist an integer $n$ and $a\notin \mathfrak{m}$ such that $a\vee \neg x^{n}=1$.
 \end{enumerate}
\end{theorem}
\begin{proof}
  \item [\ref{pmprop1}$\Rightarrow$\ref{pmprop2}:] It is an immediate consequence of Theorem \ref{contchar}.
  \item [\ref{pmprop2}$\Rightarrow$\ref{pmprop3}:] It follows by Proposition \ref{dabeta}\ref{dabeta3}.
  \item [\ref{pmprop3}$\Rightarrow$\ref{pmprop4}:] It follows by Proposition \ref{compropo}.
  \item [\ref{pmprop4}$\Rightarrow$\ref{pmprop5}:] Let $\mathfrak{p}$ be a prime filter of $\mathfrak{A}$ and $\mathfrak{m}$ a maximal filter containing $D(\mathfrak{p})$. By absurdum, assume that $\mathfrak{p}$ contained in a maximal filter $\mathfrak{n}$ distinct from $\mathfrak{m}$. This implies that $D(\mathfrak{m})$ and $D(\mathfrak{n})$ contained in $\mathfrak{m}$; a contradiction.
  \item [\ref{pmprop5}$\Rightarrow$\ref{pmprop6}:] Let $F$ be a proper filter of $\mathfrak{A}$ containing $D(\mathfrak{m})$, for some maximal filter $\mathfrak{m}$ of $\mathfrak{A}$. By absurdum, assume that $F$ contained in a maximal filter $\mathfrak{n}$ distinct from $\mathfrak{m}$. This implies that $\mathfrak{m}\subseteq \mathfrak{n}$; a contradiction.
  \item [\ref{pmprop6}$\Rightarrow$\ref{pmprop7}:] Let $\mathfrak{A}$ be a maximal filter of $\mathfrak{A}$ and $F$ be a proper filter of $\mathfrak{A}$ such that $F\veebar \mathfrak{m}=A$. Let $F\veebar D(\mathfrak{m})$ is proper. This states that $F$ contained in $\mathfrak{m}$; a contradiction.
  \item [\ref{pmprop7}$\Rightarrow$\ref{pmprop8}:] for any distinct maximal filters $\mathfrak{m}$ and $\mathfrak{n}$ of $\mathfrak{A}$, we have $D(\mathfrak{m})\veebar \mathfrak{n}=A$. This win us.
  \item [\ref{pmprop8}$\Rightarrow$\ref{pmprop9}:] It is routinely proved.
  \item [\ref{pmprop9}$\Rightarrow$\ref{pmprop10}:] Let $\mathfrak{m}$ be a maximal filter of $\mathfrak{A}$. Suppose that $\mathfrak{M}$ is a maximal filter of $\mathfrak{A}/D(\mathfrak{m})$. Using Proposition \ref{cormaxqout}, it follows that $\mathfrak{M}=\mathfrak{n}/D(\mathfrak{m})$, for some $\mathfrak{n}\in h(D(\mathfrak{m}))\cap Max(\mathfrak{A})=\{\mathfrak{m}\}$. This holds the result.
  \item [\ref{pmprop10}$\Rightarrow$\ref{pmprop11}:] It is evident by Proposition \ref{localqoiemax}.
  \item [\ref{pmprop11}$\Rightarrow$\ref{pmprop1}:] Let $\mathfrak{p}$ be a prime filter of $\mathfrak{A}$ which is contained in two distinct maximal filters $\mathfrak{m}$ and $\mathfrak{n}$ of $\mathfrak{A}$. Consider $x\in \mathfrak{n}\setminus \mathfrak{m}$. Hence, $a\vee \neg x^{n}\in \mathfrak{p}$, for some integer $n$ and $a\notin \mathfrak{m}$; a contradiction.
\end{proof}

In the following, some topological characterizations for Gelfand residuated lattices are given.
\begin{theorem}\label{p2mprop}
Let $\mathfrak{A}$ be a residuated lattice.  The following assertions are equivalent:
 \begin{enumerate}
   \item [(1) \namedlabel{p2mprop1}{(1)}] $\mathfrak{A}$ is Gelfand;
   \item [(2) \namedlabel{p2mprop2}{(2)}] any two distinct maximal filters can be separated in $Spec_{h}(\mathfrak{A})$;
   \item [(3) \namedlabel{p2mprop3}{(3)}] for any maximal filter $\mathfrak{m}$, $\mathscr{G}(\mathfrak{m})$ is closed in $Spec_{h}(\mathfrak{A})$.
 \end{enumerate}
\end{theorem}
\begin{proof}
\item [\ref{p2mprop1}$\Rightarrow$\ref{p2mprop2}:] Let $\mathfrak{m}$ and $\mathfrak{n}$ be two distinct maximal filters of $\mathfrak{A}$. So there exist $a\notin \mathfrak{m}$ and $b\notin \mathfrak{n}$ such that $a\vee b=1$. This verifies that $d(a)$ and $d(b)$ are two disjoint neighbourhoods of $\mathfrak{m}$ and $\mathfrak{n}$ in $Spec_{h}(\mathfrak{A})$, respectively.
\item [\ref{p2mprop2}$\Rightarrow$\ref{p2mprop3}:] Let $\mathfrak{m}$ be a maximal filter of $\mathfrak{A}$. Applying Proposition \ref{retractlemma}\ref{retractlemma2}, $\mathscr{G}(\mathfrak{m})$ is closed under the patch topology. Assume that $\mathfrak{p}\in \mathscr{S}\mathscr{G}(\mathfrak{m})$. So $\mathfrak{q}\subseteq \mathfrak{p}$, for some $\mathfrak{q}\in \mathscr{G}(\mathfrak{m})$. By absurdum assume that $\mathfrak{p}\notin \mathscr{G}\mathfrak{m}$. Let $\mathfrak{p}\subseteq \mathfrak{n}$, for some maximal filter $\mathfrak{n}$ of $\mathfrak{A}$. This implies that $\mathfrak{m},\mathfrak{n}\in \mathscr{S}(\mathfrak{q})$; a contradiction. Hence, the result holds by Theorem \ref{closefalzai}.
\item [\ref{p2mprop3}$\Rightarrow$\ref{p2mprop1}:] \textit{\textbf{First proof.}} Let $\mathfrak{A}$ be not Gelfand. By Theorem \ref{pmprop} there exist distinct maximal filters $\mathfrak{m}$ and $\mathfrak{n}$ of $\mathfrak{A}$ such that $D(\mathfrak{m})\subseteq \mathfrak{n}$. Suppose that $d(a)$ is a basic neighbourhood of $\mathfrak{n}$. So $a\notin \mathfrak{p}$, for some $\mathfrak{p}\in \mathscr{G}(\mathfrak{m})$. This means that $d(a)\cap \mathscr{G}(\mathfrak{m})\neq\emptyset$. Hence, $\mathfrak{n}\in cl_{h}(\mathscr{G}(\mathfrak{m}))\setminus \mathscr{G}(\mathfrak{m})$; a contradiction.\\
    \textit{\textbf{Second proof.}} Let $\mathfrak{m}$ be a maximal filter of $\mathfrak{A}$. Using Proposition \ref{omegprop}\ref{omegprop2}, it follows that $h(D(\mathfrak{m}))=hk(\mathscr{G}(\mathfrak{m}))=cl_{h}(\mathscr{G}(\mathfrak{m}))=\mathscr{G}(\mathfrak{m})$. So the result holds due to Theorem \ref{pmprop}.
\end{proof}
\begin{remark}\label{gelmaxhau}
By Proposition \ref{hulkerinstr}\ref{hulkerinstr2} and Theorem \ref{p2mprop}, if $\mathfrak{A}$ is a Gelfand residuated lattice, $Max_{h}(\mathfrak{A})$ is a $T_{4}$ space.
\end{remark}

Recall that a \textit{retraction} is a continuous mapping from a topological space into a subspace which preserves the position of all points in that subspace.
\begin{theorem}\label{gelnor}
Let $\mathfrak{A}$ be a residuated lattice. The following assertions are equivalent:
\begin{enumerate}
   \item [(1) \namedlabel{gelnor1}{(1)}] $\mathfrak{A}$ is Gelfand;
   \item [(2) \namedlabel{gelnor2}{(2)}] $Max_{h}(\mathfrak{A})$ is a retract of $Spec_{h}(\mathfrak{A})$.
 \end{enumerate}
\end{theorem}
\begin{proof}
\item [\ref{gelnor1}$\Rightarrow$\ref{gelnor2}:] For any prime filter $\mathfrak{p}$ of $\mathfrak{A}$, suppose that $\mathfrak{m_{p}}$ is the unique maximal filter containing $\mathfrak{p}$. Define $f:Spec(\mathfrak{A})\longrightarrow Max(\mathfrak{A})$ by $f(\mathfrak{p})=\mathfrak{m_{p}}$. Consider a basic closed set $\mathcal{H}=h(a)\cap Max(\mathfrak{A})$, for some $a\in A$. We claim that $\mathcal{F}=f^{\leftarrow}(\mathcal{H})$ is closed in $Spec(\mathfrak{A})$. Let $F=k(\mathcal{F})$ and $C=\bigcup \mathcal{H}$. Let $\mathfrak{p}\in cl_{h}^{Spec(\mathfrak{A})}(\mathcal{F})=hk(\mathcal{F})=h(F)$. So we have $F\subseteq \mathfrak{p}$. We have $F\subseteq C\cap \mathfrak{p}$, and so $F\cap (C\cap \mathfrak{p})^{c}=\emptyset$. Since $C^{c}$ and $\mathfrak{p}^{c}$ are $\vee$-closed subsets of $\mathfrak{A}$, so $\mathscr{C}(C^{c}\cup \mathfrak{p}^c)=\{x\vee y|x\notin C,~y\notin \mathfrak{p}\}$. Let us we denote $\mathscr{C}(C^{c}\cup \mathfrak{p}^c)$ by $\mathscr{C}$. We have $(C\cap \mathfrak{p})^{c}=C^{c}\cup \mathfrak{p}^c\subseteq \mathscr{C}$. Let $x\vee y\in F$, for some $x\in C^c$ and $y\in \mathfrak{p}^c$. Since $y\notin \mathfrak{p}$, so $y\notin F$. This follows that there exists $\mathfrak{q}\in \mathcal{F}$ such that $y\notin \mathfrak{q}$. On the other hand, $x\vee y\in F\subseteq \mathfrak{q}$, which implies that $x\in \mathfrak{q}\subseteq C$; a contradiction. So $F\cap \mathscr{C}=\emptyset$. Applying Theorem \ref{prfilth}, it ensures that there exists a prime filter $\mathfrak{q}$ containing $F$ such that $\mathfrak{q}\cap \mathscr{C}=\emptyset$. This results that $\mathfrak{q}\subseteq \mathscr{C}^{c}\subseteq C\cap \mathfrak{p}$. Assume by absurdum that $\mathscr{F}(\mathfrak{q},a)=A$. So $q\odot a^{n}=0$, for some $q\in \mathfrak{q}$ and integer $n$. Since $\mathfrak{q}\subseteq C$, so there exists some $\mathfrak{m}\in h(a)$ such that $q\in \mathfrak{m}$, but this implies that $0\in \mathfrak{m}$; a contradiction. Thus, $\mathscr{F}(\mathfrak{q},a)\subseteq \mathfrak{m}$, for some $\mathfrak{m}\in \mathcal{H}$. Hence, $\mathfrak{p}\subseteq \mathfrak{q}\subseteq \mathscr{F}(\mathfrak{q},a)\subseteq \mathfrak{m}$. This verifies that $\mathfrak{p}\in \mathcal{F}$. It states that $\mathcal{F}$ is a closed set in $Spec_{h}(\mathfrak{A})$, and so $f$ is continuous. Also, it is obvious that $f(\mathfrak{m})=\mathfrak{m}$, for any $\mathfrak{m}\in Max(\mathfrak{A})$. This win us.
\item [\ref{gelnor2}$\Rightarrow$\ref{gelnor1}:] Let $f:Spec_{h}(\mathfrak{A})\longrightarrow Max_{h}(\mathfrak{A})$ be a retraction. Let $\mathfrak{p}$ be a prime filter of $\mathfrak{A}$. Consider a maximal filter $\mathfrak{m}$ of $\mathfrak{A}$ containing $\mathfrak{p}$. Applying Lemma \ref{retractlemma}, we have $\mathfrak{m}\in cl_{h}^{Spec(\mathfrak{A})}(\mathfrak{p})$. Since $Max_{h}(\mathfrak{A})$ is $T_{1}$ and $f$ is continuous, we obtain that
\[\mathfrak{m}=f(\mathfrak{m})\in f(cl_{h}^{Spec(\mathfrak{A})}(\mathfrak{p}))\subseteq cl_{h}^{Max(\mathfrak{A})}(f(\mathfrak{p}))=\{f(\mathfrak{p})\}.\]
This shows that $\mathfrak{m}$ is the unique maximal filter of $\mathfrak{A}$ containing $\mathfrak{p}$.
\end{proof}
\begin{remark}\label{remarkretr}
By Theorem \ref{gelnor}, if $\mathfrak{A}$ is a Gelfand residuated lattice, the map $Spec(\mathfrak{A})\rightsquigarrow Max(\mathfrak{A})$, which sends any prime filter $\mathfrak{p}$ of $\mathfrak{A}$ to the unique maximal filter $\mathfrak{m_{p}}$ of $\mathfrak{A}$ containing it, is the unique retraction from $Spec_{h}(\mathfrak{A})$ into $Max_{h}(\mathfrak{A})$.
\end{remark}

By \cite{engelking1989general}, if $A$ is a compact space, $B$ is a Hausdorff space and  $f:A\longrightarrow B$ is a continuous map, then $f$ is a closed map.
\begin{theorem}\label{gelsnor}
Let $\mathfrak{A}$ be a residuated lattice. The following assertions are equivalent:
\begin{enumerate}
   \item [(1) \namedlabel{gelsnor1}{(1)}] $\mathfrak{A}$ is Gelfand;
   \item [(2) \namedlabel{gelsnor2}{(2)}] $Spec_{h}(\mathfrak{A})$ is a normal space.
 \end{enumerate}
\end{theorem}
\begin{proof}
  \item [\ref{gelsnor1}$\Rightarrow$\ref{gelsnor2}:] Using Theorem \ref{gelnor} and \textsc{Remark} \ref{remarkretr}, there exists a retraction $f:Spec_{h}(\mathfrak{A})\longrightarrow Max_{h}(\mathfrak{A})$, which sends any prime filter of $\mathfrak{A}$ to the unique maximal filter of $\mathfrak{A}$ containing $\mathfrak{p}$, for any prime filter $\mathfrak{p}$ of $\mathfrak{A}$. By \textsc{Remark} \ref{gelmaxhau}, $Max(\mathfrak{A})$ is a $T_4$ space, and so $f$ is a closed map. Let $C_1$ and $C_2$ be two disjoint closed sets in $Spec(\mathfrak{A})$, so $f(C_1)$ and $f(C_2)$ are disjoint closed sets in $Max(\mathfrak{A})$. Since $Max_{h}(\mathfrak{A})$ is normal, there exist disjoint open neighbourhoods $N_1$ and $N_2$ of $f(C_1)$ and $f(C_2)$ in $Max(\mathfrak{A})$, respectively. One can see that  $f^{\leftarrow}(N_1)$ and $f^{\leftarrow}(N_2)$ are disjoint open neighbourhoods of $C_1$ and $C_2$, respectively.
  \item [\ref{gelsnor2}$\Rightarrow$\ref{gelsnor1}:] Let $\mathfrak{m}$ be a maximal filter of $\mathfrak{A}$. Using Lemma \ref{retractlemma}\ref{retractlemma1}, it establishes that $cl_{h}^{Spec(\mathfrak{A})}(\mathfrak{m})=\mathscr{S}(\mathfrak{m})=\{\mathfrak{m}\}$. This shows that $\{\mathfrak{m}\}$ is closed in $Spec_{h}(\mathfrak{A})$. Consider distinct maximal filters  $\mathfrak{m}$ and $\mathfrak{n}$ of $\mathfrak{A}$. Since $Spec_{h}(\mathfrak{A})$ is normal, so there exist disjoint neighborhoods $N_1$ and $N_2$ for $\mathfrak{m}$ and $\mathfrak{n}$ in $Spec(\mathfrak{A})$, respectively. So the result holds by Theorem \ref{p2mprop}.
\end{proof}

Let $\mathfrak{A}$ be a residuated lattice. Consider the following relation
$\Re=\{(\mathfrak{p},\mathfrak{q})\in X^{2}\mid \mathfrak{p}\veebar \mathfrak{q}\neq A\}$ on $X=Spec(\mathfrak{A})$. Obviously, $\Re$ is reflexive and symmetric. Let $\overline{\Re}$ be the transitive closure of $\Re$.
\begin{proposition}\label{gelre}
  Let $\mathfrak{A}$ be a residuated lattice. The following assertions are equivalent:
  \begin{enumerate}
   \item [(1) \namedlabel{gelre1}{(1)}] $\mathfrak{A}$ is Gelfand;
   \item [(2) \namedlabel{gelre2}{(2)}] if $\mathfrak{m}$ is a maximal filter of $\mathfrak{A}$, then $\overline{\Re}(\mathfrak{m})=\mathscr{G}(\mathfrak{m})$.
 \end{enumerate}
\end{proposition}
\begin{proof}
  \item [\ref{gelre1}$\Rightarrow$\ref{gelre2}:] Let $\mathfrak{p}\in \overline{\Re}(\mathfrak{m})$. So there exists a finite set $\{\mathfrak{p}_1,\cdots,\mathfrak{p}_n\}$ of prime filters of $\mathfrak{A}$ with $n\geq 2$ such that $\mathfrak{p}_1=\mathfrak{p}$, $\mathfrak{p}_n=\mathfrak{m}$, and $(\mathfrak{p}_{i},\mathfrak{p}_{i+1})\in \Re$, for all $1\leq i\leq n-1$. By induction on $n$ we shall prove that $\mathfrak{p}\subseteq \mathfrak{m}$. If $n=2$, then $\mathfrak{p}\veebar \mathfrak{m}\neq A$, and this implies that $\mathfrak{p}\subseteq \mathfrak{m}$. Assume that $n>2$. We have $\mathfrak{p}_{n-2}\veebar \mathfrak{p}_{n-1}\neq A$ and $\mathfrak{p}_{n-1}\subseteq \mathfrak{m}$. So $\mathfrak{p}_{n-2}\veebar \mathfrak{p}_{n-1}\subseteq \mathfrak{n}$, for some maximal filter $\mathfrak{n}$ of $\mathfrak{A}$. Thus $\mathfrak{p}_{n-2},\mathfrak{p}_{n-1}\subseteq \mathfrak{n}$, and this states that $\mathfrak{n}=\mathfrak{m}$. Hence $\mathfrak{p}_{n-2}\subseteq \mathfrak{m}$, and so in the equivalency $(\mathfrak{p},\mathfrak{m})\in \overline{\Re}$, the number of the involved primes is reduced to $n-1$. Therefore by the induction hypothesis, $\mathfrak{p}\subseteq \mathfrak{m}$. This shows that $\overline{\Re}(\mathfrak{m})\subseteq \mathscr{G}(\mathfrak{m})$. The inverse inclusion is evident.
  \item [\ref{gelre2}$\Rightarrow$\ref{gelre1}:] Let $\mathfrak{p}$ be a prime filter of $\mathfrak{A}$ such that $\mathfrak{p}\subseteq \mathfrak{m}$ and $\mathfrak{p}\subseteq \mathfrak{n}$, for some
maximal filters $\mathfrak{m}$ and $\mathfrak{n}$ of $\mathfrak{A}$. It follows that $\mathfrak{p}\in \mathscr{G}(\mathfrak{m})\cap \mathscr{G}(\mathfrak{n})$, and so $\mathfrak{m}=\mathfrak{n}$.
\end{proof}

Let $A_{\tau}$ be a topological space, and $E$ be an equivalence relation on $A$. In the following, by $A_{\tau}/E$ we mean the quotient of the space $A_{\tau}$ modulo $E$. By \citet[p.90]{engelking1989general}, the quotient map $\pi:A_{\tau}\longrightarrow A_{\tau}/E$ is continuous, and a mapping $f$ of the quotient space $A_{\tau}/E$ to a topological space $B_{\zeta}$ is continuous if and only if the composition $f\circ \pi$ is continuous.
\begin{corollary}\label{akharincoro}
  A residuated lattice $\mathfrak{A}$ is Gelfand if and only if the map $\eta:Max_{h}(\mathfrak{A})\longrightarrow Spec_{h}(\mathfrak{A})/\overline{\Re}$, given by $\mathfrak{m}\rightsquigarrow \overline{\Re}(\mathfrak{m})$, is a homeomorphism.
\end{corollary}
\begin{proof}
  Let $\mathfrak{A}$ be a Gelfand residuated lattice. It is evident that $Spec(\mathfrak{A})/\overline{\Re}=\{\overline{\Re}(\mathfrak{m})\mid \mathfrak{m}\in Max(\mathfrak{A})\}$, and this implies that $\eta$ is a surjection. The injectivity of $\eta$ follows by Theorem \ref{gelre}, and the continuity of it follows by $\eta=\pi\circ i$, where $i$ is the inclusion map. Applying \textsc{Remark} \ref{remarkretr}, it follows that $\eta^{-1}\circ \pi$ is a retraction, and this verifies the continuity of $\eta^{-1}$, see \citet[Proposition 4.2.4]{engelking1989general}. This shows that $\eta$ is a homeomorphism. Conversely, let $\eta$ be a homeomorphism. Obviously, $\eta^{-1}\circ \pi$ is a retraction, and so $\mathfrak{A}$ is Gelfand due to Theorem \ref{gelnor}.
\end{proof}

Let $\mathfrak{A}$ be a residuated lattice. Consider the relation $\Im=\{(\mathfrak{p},\mathfrak{q})\in X^{2}\mid D(\mathfrak{p})\veebar D(\mathfrak{q})\neq A\}$ on $X=Spec(\mathfrak{A})$. Obviously, $\Im$ is reflexive and symmetric. Let $\overline{\Im}$ be the transitive closure of $\Im$.
\begin{theorem}\label{3hausnorm}
  Let $\mathfrak{A}$ be a residuated lattice. The following assertions are equivalent:
  \begin{enumerate}
   \item [(1) \namedlabel{3hausnorm1}{(1)}] $\mathfrak{A}$ is Gelfand;
   \item [(2) \namedlabel{3hausnorm2}{(2)}] for a given maximal filter $\mathfrak{m}$ of $\mathfrak{A}$, $\overline{\Im}(\mathfrak{m})=\mathscr{G}(\mathfrak{m})$.
 \end{enumerate}
\end{theorem}
\begin{proof}
  \item [\ref{3hausnorm1}$\Rightarrow$\ref{3hausnorm2}:] Let $\mathfrak{m}$ be a maximal filter of $\mathfrak{A}$. Consider $\mathfrak{p}\in \overline{\Im}(\mathfrak{m})$. So there exists a finite set $\{\mathfrak{p}_1,\cdots,\mathfrak{p}_n\}$ of elements of $Spec(\mathfrak{A})$ with $n\geq 2$ such that $\mathfrak{p}_1=\mathfrak{p}$, $\mathfrak{p}_n=\mathfrak{m}$, and $(\mathfrak{p}_{i},\mathfrak{p}_{i+1})\in \Im$, for all $1\leq i\leq n-1$. If $n=2$, then $D(\mathfrak{p})\veebar D(\mathfrak{m})\neq A$, and so $\mathfrak{p}\subseteq \mathfrak{m}$ due to Theorem \ref{pmprop}. Assume that $n>2$. We have $D(\mathfrak{p_{n-2}})\veebar D(\mathfrak{p_{n-1}})\neq A$ and $\mathfrak{p}_{n-1}\subseteq \mathfrak{m}$. This implies that $\mathfrak{p}_{n-2}\subseteq \mathfrak{m}$. This verifies that $(\mathfrak{p}_{n-2},\mathfrak{m})\in \Im$. Hence, in the equivalency $(\mathfrak{p},\mathfrak{m})\in \overline{\Im}$, the number of the involved primes is reduced to $n-1$. Therefore, $\mathfrak{p}\subseteq \mathfrak{m}$. This shows that $\overline{\Im}(\mathfrak{m})\subseteq \mathscr{G}(\mathfrak{m})$. The inverse inclusion is evident.
\item [\ref{3hausnorm2}$\Rightarrow$\ref{3hausnorm1}:] There is nothing to prove.
\end{proof}

The proof of the following corollary is analogous to the proof of Corollary \ref{akharincoro}, and so it is left to the reader.
\begin{corollary}\label{minjhomeo}
  Let $\mathfrak{A}$ be a residuated lattice. $\mathfrak{A}$ is Gelfand if and only if the map $\eta:Max_{h}(\mathfrak{A})\longrightarrow Spec_{h}(\mathfrak{A})/\overline{\Im}$, given by $\mathfrak{m}\rightsquigarrow \overline{\Im}(\mathfrak{m})$, is a homeomorphism.
\end{corollary}
\citet[p. 263]{bkouche1970couples} introduced the notion of a soft ring; ``Nous dirons qu'un anneau $A$ est \textit{mou} s'il est sans radical de Jacobson, et si tout id\'{o}al maximal est mou (c'est-\`{o}-dire max$A$ est un espace séparé). \cite{de1971commutative} introduced the notion of a \textit{pm-ring} as a unitary commutative ring in which any prime ideal is contained in a unique maximal ideal.  \citet[Theorem 1.2]{de1971commutative} proved that if  $\mathfrak{A}$ is a pm-ring, then $Max(\mathfrak{A})$ is a Hausdorff space with the Zariski topology and the converse holds provided that $\mathfrak{J(A)=N(A)}$. In spirit of \cite{bkouche1970couples}, \citet[p. 459]{levaro1975projective} considered a bit more general notion, \textit{quasi-soft rings} as a ring $\mathfrak{A}$ in which for any maximal filter $\mathfrak{m}$ of $\mathfrak{A}$ the localization map $\pi_{\mathfrak{m}}:\mathfrak{A}\longrightarrow \mathfrak{A_{m}}$ is surjective. He proved that a soft ring is quasi-soft.  Motivated by \cite{de1971commutative} the notion of a \textit{pm-lattice} is introduced by \cite{pawar1977pm} as a bounded distributive lattice in which any prime ideal is contained in a unique maximal prime ideal. \citet[Theorem 6]{pawar1977pm} proved that if $\mathfrak{A}$ is a pm-lattice, then $Max(\mathfrak{A})$ is a Hausdorff space with the Stone topology. They also stated that the converse of this theorem holds when $\mathfrak{A}$ is complemented. \citet[Theorem 4.3]{aghajani2020characterizations} showed that quasi-soft rings and Gelfand rings coincide. In the following, the above results improved and generalized to residuated lattices.
\begin{definition}
  Let $\mathfrak{A}$ be a residuated lattice. $\mathfrak{A}$ is called \textit{soft} provided that  $\mathfrak{A}$ is a semisimple  and any element of $h(Rad(\mathfrak{A}))$ is contained in a unique maximal filter of $\mathfrak{A}$.
\end{definition}

The next theorem gives some criteria for the set of maximal filters of a residuated lattice to be a Hausdorff space with the hull-kernel topology.
\begin{theorem}\label{1hausnorm}\citep{rasouli2018hull}
Let $\mathfrak{A}$ be a residuated lattice. The following assertions are equivalent:
\begin{enumerate}
   \item [(1) \namedlabel{1hausnorm1}{(1)}] $Max_{h}(\mathfrak{A})$ is a Hausdorff space;
   \item [(2) \namedlabel{1hausnorm2}{(2)}] any element of $h(Rad(\mathfrak{A}))$ is contained in a unique maximal filter of $\mathfrak{A}$;
   \item [(3) \namedlabel{1hausnorm3}{(3)}] $Max_{h}(\mathfrak{A})$ is a retract of $h(Rad(\mathfrak{A}))_{h}$;
   \item [(4) \namedlabel{1hausnorm4}{(4)}] $h(Rad(\mathfrak{A}))_{h}$ is a normal space;
   \item [(5) \namedlabel{1hausnorm5}{(5)}] $\mathscr{G}(\mathfrak{m})\cap h(Rad(\mathfrak{A}))$ is closed in $h(Rad(\mathfrak{A}))_{h}$, for any maximal filter $\mathfrak{m}$ of $\mathfrak{A}$;
   \item [(6) \namedlabel{1hausnorm6}{(6)}] for any $a,b\in A$ with $a\odot b=0$, there exist some integers $n,m$ such that $\neg a^n\vee \neg b^m\in Rad(\mathfrak{A})$.
 \end{enumerate}
\end{theorem}

The following theorem characterizes the class of soft residuated lattices.
\begin{theorem}\label{softgelfandequiv}
  Let $\mathfrak{A}$ be a residuated lattice. The following assertions are equivalent:
\begin{enumerate}
   \item [(1) \namedlabel{softgelfandequiv1}{(1)}] $\mathfrak{A}$ is soft;
   \item [(2) \namedlabel{softgelfandequiv2}{(2)}] $Max_{h}(\mathfrak{A})$ is Hausdorff and dense in $Spec_{h}(\mathfrak{A})$;
   \item [(3) \namedlabel{softgelfandequiv3}{(3)}] $\mathfrak{A}$ is Gelfand and $Rad(\mathfrak{A})=\{1\}$.
 \end{enumerate}
\end{theorem}
\begin{proof}
It is straightforward by Proposition \ref{hulkerinstr}\ref{hulkerinstr1} and Theorems \ref{gelsnor} \& \ref{1hausnorm}.
\end{proof}
\section{The pure spectrum of a Gelfand residuated lattice}\label{sec4}

This section deals with the pure spectrum of a Gelfand residuated lattice.  For the basic facts concerning pure filters of a residuated lattice, the reader is referred to \cite{rasouli2021rickart}.

Let $\mathfrak{A}$ be a residuated lattice. For any filter $F$ of $\mathfrak{A}$, we set $\sigma(F)=k(\mathscr{G}(h(F)))$.
\begin{proposition}\label{unitpro}\cite[Propositions 5.2 \& 5.4]{rasouli2021rickart}
  Let $\mathfrak{A}$ be a residuated lattice. The following assertions hold:
\begin{enumerate}
\item  [$(1)$ \namedlabel{unitpro1}{$(1)$}] $\sigma(F)=\{a\in A\mid F\veebar a^{\perp}=A\}$, for any $F\in \mathscr{F}(\mathfrak{A})$;
\item  [$(2)$ \namedlabel{unitpro2}{$(2)$}] $\sigma(F)$ is a filter of $\mathfrak{A}$ contained in $F$, for any $F\in \mathscr{F}(\mathfrak{A})$;
\item  [$(3)$ \namedlabel{unitpro3}{$(3)$}] $F\subseteq G$ implies $\sigma(F)\subseteq \sigma(G)$, for any $F,G\in \mathscr{F}(\mathfrak{A})$;
\item  [$(4)$ \namedlabel{unitpro4}{$(4)$}] $\sigma(\mathfrak{p})\subseteq D(\mathfrak{p})$, for any $\mathfrak{p}\in Spec(\mathfrak{A})$;
\item  [$(5)$ \namedlabel{unitpro5}{$(5)$}] $\sigma(\mathfrak{m})=D(\mathfrak{m})$, for any $\mathfrak{m}\in Max(\mathfrak{A})$;
\item  [$(6)$ \namedlabel{unitpro6}{$(6)$}] $\sigma(F\cap G)=\sigma(F)\cap \sigma(G)$, for any $F,G\in \mathscr{F}(\mathfrak{A})$;
\item  [$(7)$ \namedlabel{unitpro7}{$(7)$}] $\veebar_{F\in \mathcal{F}}\sigma(F)\subseteq \sigma(\veebar \mathcal{F})$, for any $\mathcal{F}\subseteq \mathscr{F}(\mathfrak{A})$.
\end{enumerate}
\end{proposition}

Let $\mathfrak{A}$ be a residuated lattice. A filter $F$ of $\mathfrak{A}$ is called \textit{pure} provided that $\sigma(F)=F$. The set of pure filters of $\mathfrak{A}$ is denoted by $\sigma(\mathfrak{A})$. It is obvious that $\{1\},A\in \sigma(\mathfrak{A})$.
\begin{proposition}\label{sigmfiltlatt}\cite[Theorem 5.7]{rasouli2021rickart}
   Let $\mathfrak{A}$ be a residuated lattice. $(\sigma(\mathfrak{A});\cap,\veebar)$ is a frame
\end{proposition}
The following result generalized and improved \cite[Theorem 1.8]{al1989pure} to residuated lattices.
\begin{proposition}\label{pureinterd}
  Let $\mathfrak{A}$ be a residuated lattice and $F$ a pure filter of $\mathfrak{A}$. We have
  \[F=\bigcap\{D(\mathfrak{m})\mid \mathfrak{m}\in Max(\mathfrak{A})\cap h(F)\}.\]
\end{proposition}
\begin{proof}
  Let $\Sigma=\{D(\mathfrak{m})\mid \mathfrak{m}\in Max(\mathfrak{A})\cap h(F)\}$. Following Proposition \ref{unitpro}\ref{unitpro5}, we obtain that $F\subseteq \bigcap\Sigma$. Assume by absurdum that $a\in \bigcap\Sigma\setminus F$. This implies that $F\veebar a^{\perp}\subseteq \mathfrak{m}$, for some maximal filter $\mathfrak{m}$ of $\mathfrak{A}$. This concludes that $a\notin \sigma(\mathfrak{m})$; a contradiction.
\end{proof}

\begin{theorem}\label{equgelchaunit}
  Let $\mathfrak{A}$ be a residuated lattice. The following assertions are equivalent:
\begin{enumerate}
\item  [(1) \namedlabel{equgelchaunit1}{(1)}] $\mathfrak{A}$ is Gelfand;
\item  [(2) \namedlabel{equgelchaunit2}{(2)}] $\sigma(F)\subseteq \mathfrak{m}$ implies $F\subseteq \mathfrak{m}$, for any $F\in \mathscr{F}(\mathfrak{A})$ and $\mathfrak{m}\in Max(\mathfrak{A})$;
\item  [(3) \namedlabel{equgelchaunit3}{(3)}] $Max(\mathfrak{A})\cap h(F)=Max(\mathfrak{A})\cap h(\sigma(F))$, for any filter $F$ of $\mathfrak{A}$;
\item  [(4) \namedlabel{equgelchaunit4}{(4)}] $Rad(F)=Rad(\sigma(F))$, for any filter $F$ of $\mathfrak{A}$;
\item  [(5) \namedlabel{equgelchaunit5}{(5)}] if $F$ and $G$ are comaximal filters of $\mathfrak{A}$, then so are $\sigma(F)$ and $\sigma(G)$;
\item  [(6) \namedlabel{equgelchaunit6}{(6)}]  $\veebar_{F\in \mathcal{F}}\sigma(F)=\sigma(\veebar \mathcal{F})$, for any $\mathcal{F}\subseteq \mathscr{F}(\mathfrak{A})$.
\end{enumerate}
\end{theorem}
\begin{proof}
\item [\ref{equgelchaunit1}$\Rightarrow$\ref{equgelchaunit2}:] Let $F$ be a filter of $\mathfrak{A}$ such that $\sigma(F)$ contained in a maximal filter $\mathfrak{m}$ of $\mathfrak{A}$. By absurdum, let $F$ not contained in $\mathfrak{m}$. So $F\veebar \mathfrak{m}=A$, and so by Theorem \ref{pmprop} and Proposition \ref{unitpro}\ref{unitpro5} follows that $F\veebar \sigma(\mathfrak{m})=A$. Hence, there exist $f\in F$ and $a\in \sigma(\mathfrak{m})$ such that $f\odot a=0$. This concludes that $\neg\neg a\in (\neg x)^{\perp}$, for some $x\in \mathfrak{m}$. So $\neg x\in \sigma(F)$; a contradiction.
\item [\ref{equgelchaunit2}$\Rightarrow$\ref{equgelchaunit3}:] It is obvious.
\item [\ref{equgelchaunit3}$\Rightarrow$\ref{equgelchaunit4}:] It is obvious.
\item [\ref{equgelchaunit4}$\Rightarrow$\ref{equgelchaunit5}:] Let $F$ and $G$ be two comaximal filter of $\mathfrak{A}$. Assume by absurdum that $\sigma(F)$ and $\sigma(G)$ are not comaximal. Thus $\sigma(F),\sigma(G)\subseteq \mathfrak{m}$, for some maximal filter $\mathfrak{m}$ of $\mathfrak{A}$. Thus $Rad(\sigma(F)),Rad(\sigma(G))\subseteq \mathfrak{m}$, and this yields that $Rad(F),Rad(G)\subseteq \mathfrak{m}$. This establishes that $F,G\subseteq \mathfrak{m}$; a contradiction.
\item [\ref{equgelchaunit5}$\Rightarrow$\ref{equgelchaunit6}:] Let $\mathcal{F}$ be a family of filters of $\mathfrak{A}$. Consider $a\in \sigma(\veebar \mathcal{F})$. So $(\veebar \mathcal{F})\veebar a^{\perp}=A$. Routinely, one can show that $(\veebar \mathcal{G})\veebar a^{\perp}=A$, for a finite subset $\mathcal{G}$ of $\mathcal{F}$. This results that $(\veebar_{G\in \mathcal{G}}\sigma(G))\veebar a^{\perp}=A$. Therefore, $a\in \sigma(\veebar_{G\in \mathcal{G}}\sigma(G))\subseteq \veebar_{F\in \mathcal{F}}\sigma(F)$.
\item [\ref{equgelchaunit6}$\Rightarrow$\ref{equgelchaunit1}:] It is an immediate consequence of Theorem \ref{pmprop} and Proposition \ref{unitpro}\ref{unitpro5}.
\end{proof}
Let $\mathfrak{A}$ be a residuated lattice. For any filter $F$ of $\mathfrak{A}$, we set
\[\rho(F)=\underline{\bigvee}\{G\in \sigma(\mathfrak{A})\mid G\subseteq F\},\]
and it is called \textit{the pure part of $F$}. Definitely, the pure part of a filter is the largest pure filter contained in it.
\begin{proposition}\label{rfilter}
  Let $\mathfrak{A}$ be a residuated lattice. The following assertions hold:
\begin{enumerate}
\item  [$(1)$ \namedlabel{rfilter1}{$(1)$}] $\rho(F)\subseteq \sigma(F)$, for any filter $F$ of $\mathfrak{A}$;
\item  [$(2)$ \namedlabel{rfilter2}{$(2)$}] The map $\rho$, given by $F\rightsquigarrow \rho(F)$, is a closure operator on $\mathscr{F}(\mathfrak{A})$, with the fixed points $\sigma(\mathfrak{A})$;
\item  [$(3)$ \namedlabel{rfilter3}{$(3)$}] The map $\rho$ is a meet-semilattice homomorphism;
\item  [$(4)$ \namedlabel{rfilter4}{$(4)$}] $F=\bigcap\{\rho(\mathfrak{m})\mid \mathfrak{m}\in Max(\mathfrak{A})\cap h(F)\}$, for any pure filter $F$ of $\mathfrak{A}$;
\item  [$(5)$ \namedlabel{rfilter5}{$(5)$}] $F=\rho(Rad(F))$, for any pure filter $F$ of $\mathfrak{A}$;
\item  [$(6)$ \namedlabel{rfilter6}{$(6)$}] $\rho(\mathfrak{p})=\rho(D(\mathfrak{p}))$, for any prime filter $\mathfrak{p}$ of $\mathfrak{A}$.
\end{enumerate}
\end{proposition}
\begin{proof}
\item [\ref{rfilter1}:] It follows by Proposition \ref{unitpro}\ref{unitpro3}.
\item [\ref{rfilter2}:] It is obvious.
\item [\ref{rfilter3}:] Let $F_1,F_2\in \mathscr{F}(\mathfrak{A})$. So $\rho(F_1\cap F_2)\subseteq \rho(F_1)\cap \rho(F_2)$, since $\rho$ is isotone. Conversely, $\rho(F_1)\cap \rho(F_2)\subseteq F_1\cap F_2$, and this implies $\rho(F_1)\cap \rho(F_2)\subseteq \rho(F_1\cap F_2)$.
\item [\ref{rfilter4}:] It is a direct result of Proposition \ref{pureinterd} and \ref{rfilter1}.
\item [\ref{rfilter5}:] Let $F$ be a pure filter of $\mathfrak{A}$. Since $\rho$ is isotone, so $F\subseteq \rho(Rad(F))$. Conversely, let $\mathfrak{m}$ be an arbitrary maximal filter of $\mathfrak{A}$ containing $F$. We get that $\rho(Rad(F))\subseteq \rho(\mathfrak{m})$, and this establishes the result due to \ref{rfilter4}.
\item [\ref{rfilter6}:] Let $\mathfrak{p}$ be a prime filter of $\mathfrak{A}$. By Proposition \ref{unitpro}\ref{unitpro4}, \ref{rfilter1} and \ref{rfilter2} follows that $\rho(\mathfrak{p})\subseteq \rho(\sigma(\mathfrak{p}))\subseteq \rho(D(\mathfrak{p}))\subseteq \rho(\mathfrak{p})$.
\end{proof}
\begin{corollary}\label{rhosigmanorg}
  Let $\mathfrak{A}$ be a Gelfand residuated lattice and $F$ a filter of $\mathfrak{A}$. Then $\rho(F)=\sigma(F)$.
\end{corollary}
\begin{proof}
  Let $a\in \sigma(F)\setminus \sigma(\sigma(F))$. So $\sigma(F)\veebar a^{\perp}\neq A$. Assume that $\sigma(F)\veebar a^{\perp}\subseteq \mathfrak{m}$, for a maximal filter $\mathfrak{m}$ of $\mathfrak{A}$. By Theorem \ref{equgelchaunit} follows that $F\subseteq \mathfrak{m}$; a contradiction.
\end{proof}

Gelfand rings are characterized in terms of pure ideals by \citet[\S 7, Theorems 31]{borceux1983algebra}. In the following theorem, these results have been improved and generalized to residuated lattices.
\begin{theorem}\label{equgelchapure}
  Let $\mathfrak{A}$ be a residuated lattice. The following assertions are equivalent:
\begin{enumerate}
\item  [(1) \namedlabel{equgelchapure1}{(1)}] $\mathfrak{A}$ is Gelfand;
\item  [(2) \namedlabel{equgelchapure2}{(2)}] $\rho(F)\subseteq \mathfrak{m}$ implies $F\subseteq \mathfrak{m}$, for any $F\in \mathscr{F}(\mathfrak{A})$ and $\mathfrak{m}\in Max(\mathfrak{A})$;
\item  [(3) \namedlabel{equgelchapure3}{(3)}] $Max(\mathfrak{A})\cap h(F)=Max(\mathfrak{A})\cap h(\rho(F))$, for any filter $F$ of $\mathfrak{A}$;
\item  [(4) \namedlabel{equgelchapure4}{(4)}] $Rad(F)=Rad(\rho(F))$, for any filter $F$ of $\mathfrak{A}$;
\item  [(5) \namedlabel{equgelchapure5}{(5)}] if $F$ and $G$ are comaximal filters of $\mathfrak{A}$, then so are $\rho(F)$ and $\rho(G)$;
\item  [(6) \namedlabel{equgelchapure6}{(6)}] if $\mathcal{F}$ is a family of filters of $\mathfrak{A}$, then $\rho(\veebar\mathcal{F})=\veebar_{F\in \mathcal{F}}\rho(F)$;
\item  [(7) \namedlabel{equgelchapure7}{(7)}] if $\mathfrak{m}$ and $\mathfrak{n}$ are maximal filters of $\mathfrak{A}$, then $\rho(\mathfrak{m})$ and $\rho(\mathfrak{n})$ are comaximal.

\end{enumerate}
\end{theorem}
\begin{proof}
\item [\ref{equgelchapure1}$\Rightarrow$\ref{equgelchapure2}:] It follows by Theorem \ref{equgelchaunit} and Proposition \ref{rhosigmanorg}.
\item [\ref{equgelchapure2}$\Rightarrow$\ref{equgelchapure3}:] It is evident.
\item [\ref{equgelchapure3}$\Rightarrow$\ref{equgelchapure4}:] It is evident.
\item [\ref{equgelchapure4}$\Rightarrow$\ref{equgelchapure5}:] Let $F$ and $G$ be two comaximal filter of $\mathfrak{A}$. Assume by absurdum that $\rho(F)$ and $\rho(G)$ are not comaximal. Thus $\rho(F),\rho(G)\subseteq \mathfrak{m}$, for some maximal filter $\mathfrak{m}$ of $\mathfrak{A}$. Thus $Rad(\rho(F)),Rad(\rho(G))\subseteq \mathfrak{m}$, and this yields that $Rad(F),Rad(G)\subseteq \mathfrak{m}$. This establishes that $F,G\subseteq \mathfrak{m}$; a contradiction.
\item [\ref{equgelchapure5}$\Rightarrow$\ref{equgelchapure6}:] Let $\mathcal{F}$ be a family of filters of $\mathfrak{A}$. Consider $a\in \rho(\veebar \mathcal{F})$. So $(\veebar \mathcal{F})\veebar a^{\perp}=A$. Routinely, one can show that $(\veebar \mathcal{G})\veebar a^{\perp}=A$, for a finite subset $\mathcal{G}$ of $\mathcal{F}$. This results that $(\veebar_{G\in \mathcal{G}}\rho(G))\veebar a^{\perp}=A$. Therefore, $a\in \sigma(\veebar_{G\in \mathcal{G}}\rho(G))\subseteq \veebar_{F\in \mathcal{F}}\rho(F)$.
\item [\ref{equgelchapure6}$\Rightarrow$\ref{equgelchapure1}:] There is nothing to prove.
\item [\ref{equgelchapure7}$\Rightarrow$\ref{equgelchapure1}:] It follows by Theorem \ref{pmprop}.
\end{proof}
In the following, the relation between pure filters and radicals in a Gelfand residuated lattice is described. This description is inspired by the one obtained for Gelfand rings \cite[\S 8, Proposition 34]{borceux1983algebra}.

Recall that a pair $(f,g)$ is called \textit{an adjunction (or isotone Galois connection)} between posets $\mathscr{A}=(A;\leq)$ and $\mathscr{B}=(B;\preccurlyeq)$, where $f:A\longrightarrow B$ and $g:B\longrightarrow A$ are two functions such that for all $a\in A$ and $b\in B$, $f(a)\leq b$ if and only if $a\preccurlyeq g(b)$. It is well-known that $(f,g)$ is an adjunction if and only if $gf$ is inflationary, $fg$ is deflationary, and $f,g$ are isotone \citep[Theorem 2]{garcia2013galois}. Also, $\mathscr{C}_{fg}=Im(f)$, in which $\mathscr{C}_{fg}$ is the set of fixed points of the kernel operator $fg$.
\begin{theorem}\label{equgelchapure}\label{rhoradgel}
  Let $\mathfrak{A}$ be a residuated lattice. The following assertions are equivalent:
\begin{enumerate}
\item  [(1) \namedlabel{rhoradgel1}{(1)}] $\mathfrak{A}$ is Gelfand;
\item  [(2) \namedlabel{rhoradgel2}{(2)}] The pair $(\rho,Rad)$ is an adjunction.
\end{enumerate}
\end{theorem}
\begin{proof}
\item [\ref{rhoradgel1}$\Rightarrow$\ref{rhoradgel2}:] Let $F$ and $G$ be two filters of $\mathfrak{A}$. Suppose first $\rho(F)\subseteq G$. Assume that $\mathfrak{m}$ is a maximal filter of $\mathfrak{A}$ containing $G$. By Theorem \ref{equgelchapure} we have $F\subseteq \mathfrak{m}$, and this yields that $F\subseteq Rad(G)$. Conversely, suppose $F\subseteq Rad(G)$. One can see that $\rho(F)\subseteq \rho(\mathfrak{m})$, for any maximal filter $\mathfrak{m}$ of $\mathfrak{A}$ containing $G$. Using Theorem \ref{equgelchapure} and Proposition \ref{rfilter}\ref{rfilter4}, we have the following sequence of formulas:
  \[
  \begin{array}{ll}
    \rho(F) & \subseteq \{\rho(\mathfrak{m})\mid \mathfrak{m}\in Max(\mathfrak{A})\cap h(G)\} \\
     & =\{\rho(\mathfrak{m})\mid \mathfrak{m}\in Max(\mathfrak{A})\cap h(\rho(G))\} \\
     & =\rho(G)\subseteq G.
  \end{array}
  \]
\item [\ref{rhoradgel2}$\Rightarrow$\ref{rhoradgel1}:] It follows by Theorem \ref{equgelchapure}.
\end{proof}
Let $\mathfrak{A}$ be a residuated lattice. A proper pure filter of $\mathfrak{A}$ is called \textit{purely-maximal} provided that it is a maximal element in the set of proper and pure filters of $\mathfrak{A}$. The set of purely-maximal filters of $\mathfrak{A}$ shall be denoted by $Max(\sigma(\mathfrak{A}))$. A proper pure filter $P$ of $\mathfrak{A}$ is called \textit{purely-prime} provided that $F_1\cap F_2\subseteq P$ implies $F_1\subseteq P$ or $F_2\subseteq P$, for any $F_1,F_2\in \sigma(\mathfrak{A})$. The set of all purely-prime filters of $\mathfrak{A}$ shall be denoted by $Spp(\mathfrak{A})$. It is obvious that $Max(\sigma(\mathfrak{A}))\subseteq Spp(\mathfrak{A})$. Zorn's lemma ensures that any proper pure filter is contained in a purely-maximal filter, and so in a purely-prime filter.
\begin{proposition}\label{r1filter}
  Let $\mathfrak{A}$ be a residuated lattice. The following assertions hold:
\begin{enumerate}
\item  [$(1)$ \namedlabel{r1filter1}{$(1)$}] $\rho(Spec(\mathfrak{A}))\in Spp(\mathfrak{A})$;
\item  [$(2)$ \namedlabel{r1filter2}{$(2)$}] $Max(\sigma(\mathfrak{A}))\subseteq \rho(Max(\mathfrak{A}))$;
\item  [$(3)$ \namedlabel{r1filter3}{$(3)$}] $F=\bigcap \{P\mid F\subseteq P\in Spp(\mathfrak{A})\}$, for any pure filter $F$ of $\mathfrak{A}$.
\end{enumerate}
\end{proposition}
\begin{proof}
\item [\ref{r1filter1}:] It is straightforward by Proposition \ref{rfilter}\ref{rfilter2}.
\item [\ref{r1filter2}:] It follows by \ref{r1filter1}.
\item [\ref{r1filter3}:] Let $\Phi=\{\mathfrak{p}\mid F\subseteq \mathfrak{p}\in Spp(\mathfrak{A})\}$ and $\Psi=\{\rho(\mathfrak{m})\mid \mathfrak{m}\in Max(\mathfrak{A})\cap h(F)\}$. Using Proposition \ref{rfilter}\ref{rfilter4} and \ref{r1filter1}, it follows that $F\subseteq \bigcap\Phi\subseteq \bigcap\Psi=F$.
\end{proof}

The following proposition characterizes the purely-maximal filters of a Gelfand residuated lattice.
\begin{proposition}\label{gelfmaxpure}
  Let $\mathfrak{A}$ be a Gelfand residuated lattice. The following assertions hold:
\begin{enumerate}
\item  [$(1)$ \namedlabel{gelfmaxpure1}{$(1)$}] $Max(\sigma(\mathfrak{A}))=\rho(Max(\mathfrak{A}))$;
\item  [$(2)$ \namedlabel{gelfmaxpure2}{$(2)$}] $Spp(\mathfrak{A})=Max(\sigma(\mathfrak{A}))$;
\item  [$(3)$ \namedlabel{gelfmaxpure3}{$(3)$}] $Spp(\mathfrak{A})=\rho(Max(\mathfrak{A}))$.
\end{enumerate}
\end{proposition}
\begin{proof}
\item [\ref{gelfmaxpure1}:] It is a direct result of Theorem \ref{equgelchapure} and Proposition \ref{r1filter}\ref{r1filter2}.
\item [\ref{gelfmaxpure2}:] Let $\mathfrak{p}$ be a purely-prime filter of $\mathfrak{A}$. Assume that $\mathfrak{m}$ is a purely-maximal filter of $\mathfrak{A}$ containing $\mathfrak{p}$. Suppose that $\mathfrak{p}\neq \mathfrak{m}$. By Theorem \ref{equgelchapure}, we have $\mathfrak{m}=\veebar_{a\in \mathfrak{m}}\rho(\mathscr{F}(a))$. Choose $a\in \mathfrak{m}$ such that $\mathscr{F}(a)\nsubseteq \mathfrak{p}$. On the other hand, $\rho(\mathscr{F}(a)\cap \rho(a^{\perp})=\{1\}\subseteq \mathfrak{p}$ which implies that $\rho(a^{\perp})\subseteq \mathfrak{p}$; a contradiction.
\item [\ref{gelfmaxpure3}:] It follows by \ref{gelfmaxpure1} and \ref{gelfmaxpure2}.
\end{proof}

For each pure filter $F$ of $\mathfrak{A}$ we set $d_{p}(F)=\{P\in Spp(\mathfrak{A})\mid F\nsubseteq P\}$. $Spp(\mathfrak{A})$ can be topologized by taking the set $\{d_{p}(F)\mid F\in \sigma(\mathfrak{A})\}$ as the open sets. The set $Spp(\mathfrak{A})$ endowed with this topology is called the \textit{pure spectrum} of $\mathfrak{A}$. It is obvious that the closed subsets of the pure spectrum are precisely of the form $h_{p}(F) =\{P\in Spp(\mathfrak{A})\mid F\subseteq P\}$, in which $F$ runs over the pure filters of $\mathfrak{A}$.

\citet[\S 8, Theorems 39]{borceux1983algebra} has shown that the pure spectrum of a ring is Hausdorff. In the following theorem, this results have been generalized to residuated lattices.
\begin{proposition}\label{gelspphau}
  Let $\mathfrak{A}$ be a Gelfand residuated lattice. $Spp(\mathfrak{A})$ is a Hausdorff space.
\end{proposition}
\begin{proof}
Let $\mathfrak{p}$ and $\mathfrak{q}$ be two distinct purely-prime filters of $\mathfrak{A}$. Applying Proposition \ref{gelfmaxpure}\ref{gelfmaxpure3}, there exists two distinct maximal filters $\mathfrak{m}$ and $\mathfrak{n}$ such that $\mathfrak{p}=\rho(\mathfrak{m})$ and $\mathfrak{q}=\rho(\mathfrak{n})$. Using Theorem \ref{pmprop}, there exist $a\notin \mathfrak{m}$ and $b\notin \mathfrak{n}$ such that $a\vee b=1$. By Theorem \ref{equgelchapure} follows that $\rho(\mathscr{F}(a))\veebar \mathfrak{m}=A$ and $\rho(\mathscr{F}(b))\veebar \mathfrak{n}=A$. Thus $\mathfrak{m}\in d_{p}(\rho(\mathscr{F}(a)))$ and $\mathfrak{n}\in d_{p}(\rho(\mathscr{F}(b)))$. By Proposition \ref{rfilter}\ref{rfilter3}, we have the following sequence of formulas:
\[
\begin{array}{ll}
  d_{p}(\rho(\mathscr{F}(a)))\cap d_{p}(\rho(\mathscr{F}(b))) & =d_{p}(\rho(\mathscr{F}(a))\cap \rho(\mathscr{F}(b))) \\
   & =d_{p}(\rho(\mathscr{F}(a)\cap \mathscr{F}(b))) \\
   & =d_{p}(\rho(\mathscr{F}(a\vee b)))=d_{p}(\{1\})=\emptyset.
\end{array}
\]
This shows that the pure spectrum of $\mathfrak{A}$ is Hausdorff.
\end{proof}

Proposition \ref{r1filter}\ref{r1filter1} leads us to a map $\rho:Spec(\mathfrak{A})\longrightarrow Spp(\mathfrak{A})$, defined by $\mathfrak{p}\rightsquigarrow \rho(\mathfrak{p})$, which is called \textit{the pure part map of $\mathfrak{A}$}. In general, this map is not injective or even surjective. However, by Proposition \ref{r1filter}\ref{r1filter2} follows that if $m$ is a purely-maximal filter of $\mathfrak{A}$, then $\rho(M)=m$, for some $M\in Max(\mathfrak{A})$.
\begin{theorem}\label{spsppconti}
   Let $\mathfrak{A}$ be a residuated lattice. The pure part map of $\mathfrak{A}$ is continuous.
\end{theorem}
\begin{proof}
Let $F$ be pure filter of $\mathfrak{A}$. Routinely, one can show that $\rho^{\leftarrow}(d_{p}(F))=d(F)$.
\end{proof}

\citet[\S 8, Proposition 40]{borceux1983algebra} has shown that the pure spectrum of a Gelfand ring is homeomorphic to its usual maximal spectrum. In the following, this result is generalized and improved to Gelfand residuated lattices.
\begin{theorem}\label{sppgelfch}
  Let $\mathfrak{A}$ be a residuated lattice. The following assertions are equivalent:
\begin{enumerate}
\item  [(1) \namedlabel{sppgelfch1}{(1)}] $\mathfrak{A}$ is Gelfand;
\item  [(2) \namedlabel{sppgelfch2}{(2)}] the map $\rho_{m}:Max_{h}(\mathfrak{A})\longrightarrow Spp(\mathfrak{A})$, given by $\mathfrak{m}\rightsquigarrow \rho(\mathfrak{m})$, is a homeomorphism.
\end{enumerate}
\end{theorem}
\begin{proof}
\item [\ref{sppgelfch1}$\Rightarrow$\ref{sppgelfch2}:] Applying Theorem \ref{spsppconti}, $\rho_{m}$ is a well-defined continuous map. Injectivity of $\rho_{m}$ follows by Theorem \ref{equgelchapure}, and surjectivity of $\rho_{m}$ follows by Proposition \ref{gelfmaxpure}. Using Propositions \ref{hulkerinstr}\ref{hulkerinstr2} and \ref{gelspphau}, it follows that $Max(\mathfrak{A})$ is compact, and $Spp(\mathfrak{A})$ is Hausdorff, respectively. This holds the result due to \citet[Theorem 3.1.13]{engelking1989general}.
\item [\ref{sppgelfch2}$\Rightarrow$\ref{sppgelfch1}:] One can see that $\rho_{m}^{-1}\circ \rho$ is a retraction from $Spec_{h}(\mathfrak{A})$ into $Max_{h}(\mathfrak{A})$. So $\mathfrak{A}$ is Gelfand due to Theorem \ref{gelnor}.
\end{proof}

The pure ideals of a commutative reduced Gelfand ring with unity are characterized in \citet[Theorems 1.8 \& 1.9]{al1989pure}. These results have been improved and generalized to residuated lattices in Theorem \ref{mppurefcl}.
\begin{theorem}\label{mppurefcl}
Let $\mathfrak{A}$ be a Gelfand residuated lattice. The pure filters of $\mathfrak{A}$ are precisely of the form $\bigcap\{D(\mathfrak{m})\mid \mathfrak{m}\in Max(\mathfrak{A})\cap \mathcal{C}\}$, in which $\mathcal{C}$ is a closed subset of $Spec_{h}(\mathfrak{A})$.
\end{theorem}
\begin{proof}
let $a\in G:=\bigcap\{D(\mathfrak{m})\mid \mathfrak{m}\in Max(\mathfrak{A})\cap \mathcal{C}\}$, in which $\mathcal{C}$ is a closed subset of $Spec_{h}(\mathfrak{A})$. So for any $\mathfrak{m}\in Max(\mathfrak{A})\cap \mathcal{C}$, we have $\mathfrak{m}\veebar a^{\perp}=A$. By absurdum, assume that $G\veebar a^{\perp}\neq A$. So $G\veebar a^{\perp}$ is contained in a maximal filter $\mathfrak{n}$. Obviously, $\mathfrak{n}\notin \mathcal{C}$. So for any $\mathfrak{m}\in  Max(\mathfrak{A})\cap \mathcal{C}$, there exist $x_{\mathfrak{m}}\notin \mathfrak{m}$ and $y_{\mathfrak{m}}\notin \mathfrak{n}$ such that $x_{\mathfrak{m}}\vee y_{\mathfrak{m}}=1$. Since $\mathcal{C}$ is stable under the specialization, so $\mathcal{C}\subseteq \bigcup_{\mathfrak{m}\in Max(\mathfrak{A})\cap \mathcal{C}} d(x_{\mathfrak{m}})$. By Proposition \ref{hulkerinstr}\ref{hulkerinstr2} follows that $\mathcal{C}$ is compact. So there exist a finite number $\mathfrak{m}_{1},\cdots,\mathfrak{m}_{n}\in Max(\mathfrak{A})\cap \mathcal{C}$ such that $\mathcal{C}\subseteq \bigcup_{i=1}^{n} d(x_{\mathfrak{m}_{i}})$. Set $x=\bigodot_{i=1}^{n} x_{\mathfrak{m}_{i}}$ and $y=\bigvee_{i=1}^{n}y_{\mathfrak{m}_{i}}$. Routinely, one can see that $x\in y^{\perp}\setminus \mathfrak{m}$, for any $\mathfrak{m}\in Max(\mathfrak{A})\cap \mathcal{C}$. This implies that $y\in G$; a contradiction. The converse follows by Proposition \ref{pureinterd}.
\end{proof}

 Let $\mathfrak{A}$ be a residuated lattice. Following \citet[Theorem 4.28]{rasouli2021rickart}, the open $\mathscr{S}$-stable subsets of $Spec_{h}(\mathfrak{A})$ are precisely of the form $d(F)$, in which $F$ runs over the pure filters of $\mathfrak{A}$. So, applying Proposition \ref{sigmfiltlatt}, the set $\tau_{\mathscr{D}}=\{d(F)\mid F\in \sigma(\mathfrak{A})\}$ forms a topology on $Spec(\mathfrak{A})$ which is called the $\mathscr{D}$-topology on $\mathfrak{A}$. It is obvious that $\tau_{h}$ is finer than $\tau_{\mathscr{D}}$. The next theorem gives some criteria for a residuated lattice to be Gelfand, inspired by the one obtained for bounded distributive lattices by \citet[Theorem 6]{al1990topological}.
\begin{theorem}\label{rickhulldmin}
  Let $\mathfrak{A}$ be a residuated lattice. The following assertions are equivalent:
\begin{enumerate}
\item  [(1) \namedlabel{rickhulldmin1}{(1)}] $\mathfrak{A}$ is Gelfand;
\item  [(2) \namedlabel{rickhulldmin2}{(2)}] the hull-kernel and $\mathscr{D}$-topology coincide on $Max(\mathfrak{A})$.
\end{enumerate}
\end{theorem}
\begin{proof}
\item [\ref{rickhulldmin1}$\Rightarrow$\ref{rickhulldmin2}:] Let $f$ be a retraction from $Spec(\mathfrak{A})$ into $Max(\mathfrak{A})$. Consider $a\in A$. Routinely, one can show that $d(a)\cap Max(\mathfrak{A})=f^{\leftarrow}(d(a)\cap Max(\mathfrak{A}))\cap Max(\mathfrak{A})$, in which $f^{\leftarrow}(d(a)\cap Max(\mathfrak{A}))$ is open $\mathscr{S}$-stable. This yields that $\tau_{\mathscr{D}}$ is finer than $\tau_{h}$ on $Max(\mathfrak{A})$.
\item [\ref{rickhulldmin2}$\Rightarrow$\ref{rickhulldmin1}:] Let $\mathfrak{p}$ be a prime filter of $\mathfrak{A}$ that is contained in two distinct maximal filters $\mathfrak{m}$ and $\mathfrak{n}$. So there exists $a\in A$ such that $\mathfrak{m}\in d(a)$ and $\mathfrak{n}\notin d(a)$. Thus $d(a)\cap Max(\mathfrak{A})=d(F)\cap Max(\mathfrak{A})$, for some pure filter $F$ of $\mathfrak{A}$. Since $\mathfrak{m}\in d(F)$, so $\mathfrak{p}\in d(F)$. This implies that $\mathfrak{n}\in d(F)$; a contradiction.
\end{proof}
\section*{\textsf{Declarations}}

\noindent\textbf{\textsf{Conflict of interest}} The authors declare that they have no conflict of
interest.

\noindent\textbf{\textsf{Ethical approval}} This article does not contain any studies with human
participants or animals performed by any of the authors.

\noindent\textbf{\textsf{Informed consent}} Informed consent was obtained from all individual
participants included in the study.

\end{document}